\documentclass[11pt]{article}

\usepackage{color}
\usepackage[centertags]{amsmath}
\usepackage{amsfonts}
\usepackage{amssymb}
\usepackage{amsthm}
\usepackage{newlfont}
\usepackage{bibspacing}
\usepackage{verbatim}

\usepackage[ps2pdf=false,pagebackref=false]{hyperref}

\hypersetup{
    colorlinks = true,
    linkcolor = magenta,
    anchorcolor = red,
    citecolor = blue,
    filecolor = red,
       urlcolor = red
}

\newtheorem{thm}{Theorem}[section]
\newtheorem*{thm*}{Theorem}
\newtheorem{cor}[thm]{Corollary}

\newtheorem{lem}[thm]{Lemma}

\newtheorem{prop}[thm]{Proposition}
\newtheorem*{prop*}{Proposition}

\newtheorem*{conj*}{Conjecture}

\newtheorem*{dfn*}{Definition}
\theoremstyle{definition}
\newtheorem{rem}[thm]{\textbf{Remark}}
\newtheorem{example}[thm]{Example}
\newtheorem*{rmk*}{Remark}
\newtheorem*{fact*}{Fact}

\theoremstyle{proof}

\newcommand{\norm}[1]{\left\Vert#1\right\Vert}

\newcommand{\abs}[1]{\left\vert#1\right\vert}
\newcommand{\set}[1]{\left\{#1\right\}}
\newcommand{\brac}[1]{\left(#1\right)}
\newcommand{\scalar}[1]{\left \langle #1 \right \rangle}
\newcommand{\sscalar}[1]{\langle #1 \rangle}
\newcommand{\Real}{\mathbb{R}}

\newcommand{\eps}{\varepsilon}
\newcommand{\E}{\mathcal{E}}

\newcommand{\D}{\mathcal{D}}

\newcommand{\grad}{\nabla}

\newlength{\defbaselineskip}
\newcommand{\setlinespacing}[1]           {\setlength{\baselineskip}{#1 \defbaselineskip}}

\numberwithin{equation}{section}

\newcommand{\B}{\mathcal{B}}

\begin{document}

\title{A Generalization of Caffarelli's Contraction Theorem via (reverse) Heat Flow}
\author{Young-Heon Kim\textsuperscript{1} and Emanuel Milman\textsuperscript{2}}

\footnotetext[1]{
Department of Mathematics,
University of British Columbia,
Vancouver, BC
Canada.
Email: yhkim@math.ubc.ca.}

\footnotetext[2]{
Department of Mathematics,
University of Toronto,
40 St. George Street,
Toronto, Ontario M5S 2E4,
Canada.
Phone: +1-416-946-5440.
Email: emilman@math.toronto.edu.\\
Both authors were partially supported by the Institute for Advanced Study through NSF grant DMS-0635607.
YHK is also supported by Canadian NSERC discovery grant 371642-09.}

\maketitle

\begin{abstract}
A theorem of L. Caffarelli implies the existence of a map, pushing forward a source Gaussian measure to a target measure which is more log-concave than the source one, which contracts Euclidean distance (in fact, Caffarelli showed that the optimal-transport Brenier map $T_{opt}$ is a contraction in this case).
We generalize this result to more general source and target measures, using a condition on the third derivative of the potential, by providing two different proofs. The first uses a
map $T$, whose inverse is constructed as a flow along an advection field associated to an appropriate heat-diffusion process. The contraction property is then reduced to showing that log-concavity is preserved along the corresponding diffusion semi-group, by using a maximum principle for parabolic PDE. In particular, Caffarelli's original result immediately follows by using the Ornstein-Uhlenbeck process and the Pr\'ekopa--Leindler Theorem. The second uses the map $T_{opt}$ by generalizing Caffarelli's argument, employing in addition further results of Caffarelli. As applications, we obtain new correlation and isoperimetric inequalities.
\end{abstract}

\section{Introduction}

The starting point of this work is the following ``Contraction Theorem" of L. Caffarelli \cite{CaffarelliContraction}:

\begin{thm*}[Caffarelli]
Let $\mu = \exp(-Q(x)) dx$ and $\nu = \exp(-(Q(x) + V(x))) dx$ denote two Borel probability measures on Euclidean space $(\Real^n,\abs{\cdot})$,
where $Q$ denotes a quadratic function, i.e.
\begin{equation} \label{eq:quad}
Q(x) = \scalar{Ax,x} + \scalar{b,x} + c ~,
\end{equation}
with $A$ positive-definite, and $V$ is a convex function. Then the Brenier optimal-transport map $T = T_{opt}$ pushing forward $\mu$ onto $\nu$ is a contraction:
\[
\forall x,y \in \Real^n \;\;\; |T(x) - T(y)| \leq |x - y| ~.
\]
\end{thm*}

Let us recall some of the notions used above. A Borel map $T$ is said to push-forward $\mu$ onto $\nu$, denoted $T_*(\mu) = \nu$, if $\nu(A) = \mu(T^{-1}(A))$ for any Borel set $A$. Among all such maps $T$, it is natural to minimize the squared-distance transport cost: $W_2^2(\mu,\nu):= \inf_{T_*(\mu) = \nu} \int |T(x) - x|^2 d\mu(x)$ - this is precisely the Monge (or Monge-Kantorovich) problem for a quadratic cost. The Brenier map $T_{opt} : \Real^n \rightarrow \Real^n$ pushing forward $\mu$ onto $\nu$ is the $\mu$-a.e. unique map for which the latter infimum is attained; it is precisely characterized by the property of being the gradient of a convex function $\varphi: \Real^n \rightarrow \Real$, as first proved by Y. Brenier \cite{BrenierMap}. It is known that the optimal-transport distance $W_2$ metrizes the Wasserstein space $W_2(\Real^n)$ of square integrable Borel probability measures on $\Real^n$ equipped with a suitable weak topology. We refer to \cite{VillaniTopicsInOptimalTransport,VillaniOldAndNew} for a comprehensive account on this and related topics.

\subsection{Main Result}

Fix an orthogonal decomposition of $(\Real^n,|\cdot|)$ into subspaces $\set{E_i}_{i=0}^k$.
\begin{dfn*}
We will say that a function $F: \Real^n \rightarrow \Real$ \emph{satisfies our symmetry assumptions} if it is invariant under the action of the subgroup $O(E_1,\ldots,E_k) := 1 \times O(E_1) \times \ldots O(E_k)$ of the orthogonal group $O(n)$, or equivalently, if:
\begin{equation} \label{eq:symmetry}
\text{$\exists \,  \Phi : \Real^{dim E_0 + k} \rightarrow \Real \;\;\text{so that}\;\; F(x) = \Phi(Proj_{E_0} x,|Proj_{E_1} x|,\ldots,|Proj_{E_k} x|)$} ~.
\end{equation}
We will similarly say that a map $T : \Real^n \rightarrow \Real^n$ \emph{satisfies our symmetry assumptions} if it commutes with the action of the latter subgroup.
\end{dfn*}

Our main result generalizes Caffarelli's Theorem as follows:

\begin{thm} \label{thm:main1}
Let $\mu = \exp(-U(x)) dx$ and $\nu = \exp(-(U(x) + V(x))) dx$, denote two Borel probability measures on Euclidean space $(\Real^n,\abs{\cdot})$. Assume that $U \in C^{3,\alpha}_{loc}(\Real^n)$ ($\alpha > 0$) is a convex function of the form:
\begin{equation} \label{eq:U}
U(x) = Q(Proj_{E_0} x) + \sum_{i=1}^k \rho_i(|Proj_{E_i} x|) ~,~ \forall i = 1 \ldots k ~ ~ \rho'''_i \leq 0 \text{ on $\Real_+$ }~,
\end{equation}
where $Q : E_0 \rightarrow \Real$ is a quadratic function as in (\ref{eq:quad}), and that $V : \Real^n \rightarrow \Real$ is convex and satisfies
our symmetry assumptions (\ref{eq:symmetry}).
Then there exists a map $T: \Real^n \rightarrow \Real^n$ pushing forward $\mu$ onto $\nu$ and satisfying our symmetry assumptions which is a contraction.
\end{thm}

\begin{rmk*}
The smoothness assumption on $U$ above is immaterial, and may be dispensed if $\mu$ is approximated (say, in total-variation distance) by measures $\set{\mu_l}$ which satisfy the conditions
of Theorem \ref{thm:main1} (see Lemma \ref{lem:approx}). A prototypical example where this applies is for the functions $\rho_i(x) = |x|^{p_i}$, $p_i \in [1,2]$. The same comment applies for $V$, so it is enough to prove the theorem for smooth $U,V$, and conclude by a compactness argument detailed in Section \ref{sec3}.
\end{rmk*}

\medskip

The general formulation of Theorem \ref{thm:main1} interpolates between the following extremal cases:
\begin{itemize}
\item
$\mu$ is a product measure and $V$ is ``unconditional":
\[
\text{$U(x) = \sum_{i=1}^n \rho_i(|x_i|)$ with $\rho_i''' \leq 0$ and $V(x_1,\ldots,x_n) = V(\pm x_1,\ldots,\pm x_n)$ are convex} ~.
\]
\item
$U$ and $V$ are both radial:
\[
\text{$U(x) = \rho(|x|)$ with $\rho''' \leq 0$ and $V(x) = \Phi(|x|)$ are convex} ~.
\]
\item
$U$ is quadratic and $V$ is an arbitrary convex function.
\end{itemize}
We shall be mainly interested in the first case, since the third one follows immediately from Caffarelli's result, and the second one may be easily obtained using a one dimensional argument reproducing Caffarelli's original proof, as described in Section \ref{sec:revisit}. However, for some of the applications presented in this work, the case when $0 < dim E_0 < n$ is the most interesting. We also remark that Caffarelli's theorem has recently been generalized in other directions by Valdimarsson \cite{ValdimarssonGeneralizedCaffarreli} and Kolesnikov \cite{KolesnikovGeneralizedCaffarelli}.

\subsection{The Construction}\label{SS:construction of T}

As opposed to the non-constructive optimal-transport map $T_{opt}$, our map $T$ is obtained as a limit of diffeomorphisms $\set{T_t}_{t \geq 0}$, constructed as a (reverse) flow along an advection field generated by an appropriate heat diffusion process. Let $L$ denote the following second-order differential operator:
\begin{equation} \label{eq:L-def}
L = \exp(U) \; \nabla \cdot (\exp(-U) \nabla) = \Delta - \scalar{\nabla,\nabla U} ~,
\end{equation}
and let $P^U_t := \exp(t L) : L_\infty(\Real^n) \rightarrow L_\infty(\Real^n)$ denote the associated diffusion semi-group, characterized as solving the parabolic equation:
\begin{equation} \label{eq:semi-group}
\frac{d}{dt} P^U_t(f) = L(P^U_t(f)) ~,~ P^U_0(f) = f ~ ~ \text{(for smooth bounded functions $f$)} ~.
\end{equation}
The latter is simply the usual heat-equation with an additional first-order drift term, also known as the (linear) Fokker-Planck equation. Its invariant measure is easily checked to be $\mu = \exp(-U(x)) dx$:
\begin{equation} \label{eq:L-prop}
\int L(f) g d\mu = - \int \scalar{\nabla f, \nabla g} d\mu = \int f L(g) d\mu ~,~ \int P^U_t(f) g d\mu = \int f P^U_t(g) d\mu ~.
\end{equation}
In particular, $-L$ becomes a self-adjoint positive semi-definite operator on an appropriate dense subspace of $L_2(\mu)$.
Since $\mu$ is a log-concave probability measure, it is known that $-L$ has a non-trivial spectral-gap, from which it follows by the Spectral Theorem that $P_t^U(f) \rightarrow_{t \rightarrow \infty} \int f d\mu$ in a rather strong sense (see Section \ref{sec3}). Defining:
\begin{equation} \label{eq:nu-t-def}
\nu_t := P^U_t(\exp(-V)) \mu ~,
\end{equation}
it follows in particular that $\nu_0 = \nu$ and $\nu_t \rightarrow_{t \rightarrow \infty} \mu$, so $\set{\nu_t}$ naturally interpolate between $\nu$ and $\mu$.
We will show how to construct diffeomorphisms $\set{T_t}_{t \geq 0}$, so that each $T_t$ is a contraction satisfying our symmetry assumptions which pushes forward $\nu_t$ onto $\nu$. Theorem \ref{thm:main1} then follows by a compactness argument, ensuring that $\set{T_t}$ converge appropriately to our desired map $T$.

Our construction is in fact for the inverse-maps $S_t := T_t^{-1}$, pushing forward $\nu$ onto $\nu_t$. These diffeomorphisms are constructed as a flow along
a (time-dependent) advection field $W_t$ induced by our diffusion:
\begin{equation} \label{eq:vector-flow}
\frac{d}{dt} S_t(x) = W_t(S_t(x)) ~,~ S_0 = Id ~.
\end{equation}
To choose a consistent $W_t$, we use the well-known Continuity Equation (see e.g. \cite{VillaniTopicsInOptimalTransport}):
\[
\frac{d}{dt} \nu_t + \nabla \cdot (\nu_t W_t) = 0 ~,
\]
which allows us to pass from the Lagrangian view point (\ref{eq:vector-flow}) to an Eulerian one. We conclude using (\ref{eq:nu-t-def}) that:
\[
\frac{d}{dt} P^U_t(\exp(-V)) = - \exp(U) \; \nabla \cdot (\exp(-U) P^U_t(\exp(-V)) W_t ) ~,
\]
and to make this consistent with (\ref{eq:L-def}) and (\ref{eq:semi-group}), we choose:
\begin{equation} \label{eq:W-t-def}
W_t := - \nabla \log P^U_t(\exp(-V)) ~.
\end{equation}

It remains to show that the maps $S_t$ are expansions, i.e. $|S_t(x) - S_t(y)| \geq |x-y|$. Being diffeomorphisms, this is equivalent
to requiring that the maps are expansions \emph{locally}:
\[
(DS_t)^* DS_t \geq Id ~.
\]
Differentiating this inequality in $t$ and using (\ref{eq:vector-flow}), we see that it suffices to show that $D W_t + (D W_t)^*\geq 0$ for all $t \geq 0$. By (\ref{eq:W-t-def}), this is equivalent to showing that:
\[
 - D^2 \log P^U_t(\exp(-V)) \geq 0 \;\;\; \forall t \geq 0 ~.
\]

\subsection{The Reduction}

This is formulated in the following result, which we believe is of independent interest:
\begin{thm} \label{thm:main2}
Under the assumptions of Theorem \ref{thm:main1}, $P^U_t$ preserves the log-concavity of $\exp(-V)$. In other words, $-\log P^U_t(\exp(-V))$ is a convex function for all $t \geq 0$.
\end{thm}

It should be noted that by a result of A. Kolesnikov \cite{KolesnikovPreservingLogConcavity} (see also \cite{LionsMusielaPreservingConvexityCRAS}, and compare with \cite{IshigeSalaniQuasiConcavityIsNotPreserved}), the only smooth linear diffusion processes (\ref{eq:semi-group}) with generator $L = A(x) \nabla^2 + b(x) \nabla$ which preserve the log-concavity of $\exp(-V)$ for \emph{arbitrary} convex functions $V$, are precisely the Ornstein-Uhlenbeck processes, given by a constant valued matrix $A$ and an affine map $b$ (for our generator (\ref{eq:L-def}), this corresponds to quadratic potentials $U=Q$). That the Ornstein-Uhlenbeck processes preserve log-concavity is well known, and may be easily seen using the Mehler formula and the Prekop\'a-Leindler Theorem (e.g. \cite{HargeGCCForEllipsoid}); together with our construction above, this already provides an alternative proof of Caffarelli's Contraction Theorem (with some other map $T$).
By \emph{restricting} to convex functions $V$ having certain symmetries, as in Theorem \ref{thm:main2}, we are able to show that log-concavity is preserved for generators with \emph{more general} potentials $U$.

The proof of Theorem \ref{thm:main2} is based on parabolic PDE methods and in particular the maximum principle (see \cite{KawohlBook,LeeVazquezParabolicNonlinearPDE,GrecoKawohl} and the references therein). Let us give a very heuristic outline of the proof. After assuming that $V$ is smooth enough and strictly convex, and restricting the problem onto a smooth, bounded and strictly convex domain by imposing zero Dirichlet boundary conditions, we proceed in the contrapositive. Assume that $V = V(x,t)$ does not remain strictly convex, and argue that there will be a first time $t_0 > 0$ when this fails; this step is the most delicate in all of the proof and requires very careful justification, a point that has been omitted in many previous works on concavity properties of solutions to parabolic PDE. The strict convexity of the boundary guarantees that the minimum of $D^2_{e,e}V(x,t_0)$ will be attained in an interior point $x_0$ and some direction $e$. Since this will be a local minimum, this implies on one hand that $(d/dt - \Delta)(D^2_{e,e}V)(x_0,t_0) \leq 0$. On the other hand, using that $D D_e V = 0$ and $D D^2_{e,e} V = 0$ at $(x_0,t_0)$, a calculation shows that:
\[
\left . \brac{(d/dt - \Delta)(D^2_{e,e}V)} \right |_{(x_0,t_0)} = - \left (D^3 U) \right|_{x_0} (e,e,\grad V(x_0,t_0)) ~.
\]
At time $t=t_0$, $V(\cdot,t)$ is still assumed to be convex, and our geometric structural and symmetry assumptions on $U$ and $V$ were precisely designed to guarantee that the latter expression be non-negative. Massaging this argument a little more, we obtain a contradiction, thereby concluding the proof.
We emphasize again that key to our approach
is an analysis at the very first time $t_0>0$ when things may go wrong - a triviality for the usual application of the maximum principle for a (uniformly continuous) function on a bounded parabolic domain, but a genuine issue when applied to its second derivatives, which may not be uniformly continuous up to the boundary.

\subsection{Applications}

Besides the applications provided in his original paper \cite{CaffarelliContraction}, Caffarelli's Contraction Theorem has found numerous applications in various fields, serving as a tool to transfer isoperimetric inequalities, obtaining correlation inequalities, and more (see e.g. \cite{CorderoMassTransportAndGaussianInqs,CFM-BConjecture,HargeCorrelationInqs,KlartagMarginalsOfInequalities}). Most of these applications only use the fact that there exists \emph{some} contracting map pushing forward one measure onto another, without employing the additional information that this map is the \emph{Brenier} map, i.e. the gradient of a convex function. Consequently, it is a mere exercise to repeat the corresponding proofs in our more general setting, replacing Caffarelli's Theorem with Theorem \ref{thm:main1}, and thereby extending these applications. We will not go through all of these in this work, but rather indicate several selected applications pertaining to correlation inequalities, extending in particular some known results regarding the Gaussian Correlation Conjecture (described in Section \ref{sec:apps}) to our setup, following an argument of Dario Cordero-Erausquin \cite{CorderoMassTransportAndGaussianInqs}. We will also briefly indicate how to obtain new isoperimetric inequalities.

\subsection{Afterthoughts}

After understanding how to extend Caffarelli's Contraction Theorem using our heat-induced flow and proving Theorem \ref{thm:main1}, we revisited Caffarelli's original argument from \cite{CaffarelliContraction}, and observed:
\begin{thm} \label{thm:main3}
Theorem \ref{thm:main1} is also valid when replacing $T$ with the Brenier optimal-transport map $T_{opt}$ pushing forward $\mu$ onto $\nu$.
\end{thm}
For the proof of Theorem \ref{thm:main3}, which is based on Caffarelli's own proof, we require an additional
ingredient from \cite{CaffarelliContraction} in the form of Theorem \ref{thm:ingred}, described in Section \ref{sec:revisit}. Roughly speaking, Caffarelli's argument is oblivious to the quadratic part of $U$, and for the non-quadratic part on $E_0^\perp$, reduces under our assumptions the task of showing that $T_{opt}$ is a contraction, to showing that it is a contraction \emph{with respect to the origin}. It is this latter property which is verified using Theorem \ref{thm:ingred}.

\medskip
In Section \ref{sec:comparing}, we compare between the two maps $T$ (as constructed in Subsection~\ref{SS:construction of T}) and $T_{opt}$. It is not hard to verify that the path $[0,\infty) \ni t \mapsto S_t$ of our interpolating diffeomorphisms does not coincide in general with the path $[0,1) \ni s \mapsto (1-s) Id + s S_{opt}$ of optimal interpolating maps, where $S_{opt} = T_{opt}^{-1}$ denotes the Brenier map pushing forward $\nu$ onto $\mu$. Indeed, our diffusion process may be seen as the gradient flow for the entropy functional $H(\nu_t | \mu)$ on the Wasserstein space $W_2(\Real^n)$ equipped with an appropriate Riemannian structure (Otto and Villani \cite{OttoVillaniHWI}, see also Jordan--Kinderlehrer--Otto \cite{JKO-FokkerPlanckAsGradientDescent}); optimal-transport, on the other hand, corresponds to moving along the geodesic between $\nu$ and $\mu$ in $W_2(\Real^n)$, i.e. gradient flow for the distance squared functional $W_2^2(\nu_t ,\mu)$.
Consequently, we believe that the limiting maps $T$ and $T_{opt}$ are in general different, although we have not been able to exclude the possibility that they coincide.
The assumptions of Theorem \ref{thm:main1} were precisely designed to ensure that $T$ contracts distances, but it is quite surprising that exactly the same assumptions imply (for seemingly different reasons!) the same for $T_{opt}$.

When comparing these two approaches, it is worth pointing out that our diffusion approach only relies on classical regularity results for linear parabolic PDEs, whereas analyzing the optimal-transport map requires Caffarelli's deeper regularity results for the fully-nonlinear Monge-Amp\`ere equation (see \cite{CaffarelliStrictlyConvexIsHolder,CaffarelliRegularity} and the references therein); consequently, the former approach may lend itself to further generalization, in particular to setups where the latter regularity results for the Brenier--McCann optimal-transport map are unavailable, or alternatively, known to be false, as in the Riemannian-manifold setting (see \cite{VillaniOldAndNew}).

\subsection{Organization}

The rest of this work is organized as follows. In Section \ref{sec2} we provide a complete proof of Theorem \ref{thm:main2}. In Section \ref{sec3}, we rigorously justify the proof of Theorem \ref{thm:main1} described above, providing the (few) missing details in the above construction. In Section \ref{sec:apps} we present some applications of Theorem \ref{thm:main1}. In Section \ref{sec:revisit}, we revisit Caffarelli's argument and provide an alternative proof of Theorem \ref{thm:main1} for the Brenier map $T_{opt}$ itself. Lastly, in Section \ref{sec:comparing}, we compare between the two
maps $T$ and $T_{opt}$, and conclude with some final remarks.

\medskip

\noindent \textbf{Acknowledgements.} We gratefully acknowledge the support of the Institute for Advanced Study, where this work was initiated,
and thank Jean Bourgain and Tom Spencer for their support and interest. We would also like to thank Cedric Villani for his interest and for providing several helpful references during this work, Almut Burchard, Bo'az Klartag and Robert McCann for their interest and illuminating remarks, and Dominic Dotterrer for remarks on terminology. We also thank Haim Brezis, Bob Jerrard, Ki-Ahm Lee, Alessandra Lunardi and Vladimir Maz'ya for their patient help with Proposition \ref{prop:regularity}, and Bernd Kawohl for additional references. Final thanks go out to the anonymous referees, for helpful suggestions which have improved the presentation of this work.

\section{Proof of Theorem \ref{thm:main2}} \label{sec2}

This section is dedicated to the proof of Theorem \ref{thm:main2}, from which Theorem \ref{thm:main1} easily follows, as explained in the Introduction, and rigorously verified in Section \ref{sec3}.
We begin by setting up the notation throughout the paper. Our basic reference is \cite{LadyParabolicBook}, even though our notation varies slightly from the notation used there. We will use $D$ and $\nabla$ interchangeably to denote the derivative operator in $\Real^n$.
Given an non-negative integer $k$, we denote by $C^{k}(\Sigma)$ the space of real-valued functions on $\Sigma \subset \Real^n$ with continuous derivatives $D^a f$, for every multi-index $a$ of order $|a| \leq k$, equipped with the usual maximum norm:
\[
\norm{f}_{C^{k}(\Sigma)} := \sum_{|a| \leq k} \sup_{x \in \Sigma} |D^{a} f(x)| ~.
\]
Similarly, the space $C^{k+\alpha}(\Sigma) = C^{k,\alpha}(\Sigma)$ denotes the subspace of functions whose $k$-th order derivatives are H\"{o}lder continuous of order $\alpha \in (0,1]$, equipped with the norm:
\[
\norm{f}_{C^{k,\alpha}(\Sigma)} := \norm{f}_{C^{k}(\Sigma)} + \sum_{|a| = k} \sup_{x \neq y \in \Sigma} \frac{|D^a f(x) - D^a f(y)|}{|x-y|^{\alpha}} ~.
\]
We will say that a continuous function is H\"{o}lder continuous of order 0, in which case $C^{k}(\Sigma)$ indeed coincides with $C^{k,0}(\Sigma)$.

When $\Sigma = \Omega \times \Theta$ is a product domain consisting of space $x \in \Omega$ and time $t \in \Theta$ components, we will denote by
$C^{k \times l}(\Omega \times \Theta)$ the space of real-valued functions $f$ with continuous (in $\Sigma$) space derivatives $D_x^a$ of order $|a| \leq k$ and time derivatives $D_t^s$ of order $s \leq l$, equipped with the norm:
\[
\norm{f}_{C^{k \times l}(\Omega \times \Theta)} := \sum_{{|a| \leq k}} \sup_{z \in \Sigma} |D^{a}_x f(z)| + \sum_{s=0}^{l} \sup_{z \in \Sigma} |D^{s}_t f(z)|
~.
\]
We will also denote by $C^{(\beta ; \beta/2)}(\Omega \times \Theta)$ the space of real-valued functions $f$ on $\Sigma$ such that for every integer $r,s \geq 0$ with $r + 2s \leq \beta$ and $|a|=r$, $D_x^{a} D_t^{s} f$ is H\"{o}lder continuous in $x$ of order $\min(\beta-(r+2s),1)$ and in $t$ of order $\min(\beta/2-(r/2 +s),1)$. The natural norm on this space is given by:
\begin{eqnarray*}
\norm{f}_{C^{(\beta ; \beta/2)}(\Omega \times \Theta)} & := & \sum_{r + 2s \leq  \lfloor \beta \rfloor} \sum_{|a|=r} \sup_{z \in \Sigma} |D_x^{a} D_t^{s} f(z)| + \\
& + & \sum_{r + 2s =  \lfloor \beta \rfloor} \sum_{|a|=r} \sup_{x_1 \neq x_2 \in \Omega , t \in \Theta} \frac{|D_x^{a} D_t^{s} f(x_1,t) - D_x^{a} D_t^{s} f(x_2,t)|}{|x_1 - x_2|^{\beta - (r+2s)}} \\
& + & \sum_{ \lfloor \beta \rfloor -1 \leq r + 2s \leq  \lfloor \beta \rfloor} \sum_{|a|=r}  \sup_{x \in \Omega , t_1 \neq t_2 \in \Theta} \frac{|D_x^{a} D_t^{s} f(x,t_1) - D_x^{a} D_t^{s} f(x,t_2)|}{|t_1 - t_2|^{\beta/2 - (r/2+s)}} ~.
\end{eqnarray*}
Lastly, we will denote by $W_p^{(2l;l)}(\Omega \times \Theta)$ for $p \in [1,\infty]$ and $l$ a non-negative integer, the space of functions $f$ on $\Omega \times \Theta$ so that for any integer $r,s \geq 0$ with $r + 2s \leq l$ and $|a|=r$, the distributional derivatives $D_x^{a} D_t^{s} f$ are in $L_p(\Omega \times \Theta)$ (this space is equipped with its usual Sobolev norm, which we will not require explicitly).

Finally, we let $F_{loc}(\Sigma)$ denote the space of functions belonging to $F(\Pi)$ for all compact subsets $\Pi$ of $\Sigma$.

\subsection{Reduction to smooth $V$} \label{subsec:smooth}

Let us start by summarizing several well-known properties of the semi-group $\set{P_t^U}_{t \geq 0}$. From the classical theory of parabolic equations, it follows that for each $t \geq 0$, $P_t^U$ acts linearly on the space $\B(\Real^n)$ of smooth bounded functions on $\Real^n$ to itself (indeed, there exists a unique solution of (\ref{eq:semi-group}) in the class of bounded functions), and hence is a semi-group $P^U_t \circ P^U_s = P^U_{t+s}$. Moreover, by the maximum principle, it follows that $\norm{P_t^U(f)}_{L_\infty} \leq \norm{f}_{L_\infty}$ and that $P_t^U(f) \geq 0$ for any $f \geq 0$ in $\B(\Real^n)$. Since $\int P_t^U(f) d\mu = \int f d\mu$, as easily checked by differentiating in $t$ and using (\ref{eq:semi-group}), it follows by interpolation that $\norm{P_t^U(f)}_{L_p(\mu)} \leq \norm{f}_{L_p(\mu)}$ for all $p \in [1,\infty]$. Consequently, the action of $P^U_t$ extends to all of the $L_p(\mu)$ spaces, clarifying the statement of Theorem \ref{thm:main2}.

It follows immediately that it is enough to prove Theorem \ref{thm:main2} for smooth functions $V$. Indeed, any convex function $V : \Real^n \rightarrow \Real \cup \set{+\infty}$ may be pointwise approximated from below by a non-decreasing sequence of smooth convex functions $V_m : \Real^n \rightarrow \Real$, which may be chosen to preserve any symmetry properties satisfied by $V$. In particular, $\exp(-V_m)$ tends to $\exp(-V)$ in $L_1(\mu)$, and so $V_m + c_m$ satisfy the assumptions of Theorem \ref{thm:main2}, where $c_m \rightarrow 0$ denote normalization constants ensuring that $\exp(-(V_m + c_m)) \mu$ are probability measures. Consequently $P_t^U(\exp(-V_m))$ tends to $P_t^U(\exp(-V))$ in $L_1(\mu)$, and since the sequence $P_t^U(\exp(-V_m))$ is pointwise non-increasing (using the positivity of $P_t^U$), it follows that there exists a pointwise limit which coincides with $P_t^U(\exp(-V))$ in $L_1(\mu)$. By assuming that Theorem \ref{thm:main2} holds for smooth functions, it follows that $P_t^U(\exp(-V_m))$ are log-concave:
\[
P_t^U(\exp(-V_m))\brac{\frac{x+y}{2}} \geq P_t^U(\exp(-V_m))(x)^{\frac{1}{2}} P_t^U(\exp(-V_m))(y)^{\frac{1}{2}} \;\;\; \forall x,y \in \Real^n ~,
\]
and this is clearly preserved under pointwise limit. The reduction to the case that $V$ is smooth is complete.

\subsection{Reduction to vanishing Dirichlet boundary conditions}

Let $B(R)$ denote the open Euclidean ball in $\Real^n$ of radius $R$ centered at the origin, and let $\chi : [0,1] \rightarrow [0,1]$ denote a smooth log-concave (non-increasing) function so that $\chi|_{[0,1)} > 0$, $\chi|_{[0,1/2]} \equiv 1$ and $\chi(1)=0$.

\begin{prop} \label{prop:Dirichlet}
Let $U \in C^{1,\alpha}_{loc}(\Real^n)$, $V\in C^{2,\alpha}_{loc}(\Real^n)$ and $\exp(-V) \in C^0(\Real^n)$. Assume that for any $R,T > 0$, the solution $f_R(x,t)$ to the parabolic equation:
\[
\frac{d}{dt} f_R = \Delta f_R - \scalar{\nabla f_R,\nabla U} ~,~  f_R(x,0) = \exp(-V(x)) \chi(|x|/R) ~ , ~  (x,t) \in B(R) \times [0,T] ~,
\]
with \emph{vanishing Dirichlet boundary conditions}:
\[
f|_{\partial B_R \times [0,T]} \equiv 0 ~,
\]
is spatially log-concave on $B(R)$ for any $t \in [0,T]$. Then the (unique) bounded solution $f(x,t)$ to the Cauchy problem:
\begin{equation} \label{eq:Cauchy}
\frac{d}{dt} f = \Delta f - \scalar{\nabla f,\nabla U} ~,~ f(x,0) = \exp(-V(x)) ~, ~ (x,t) \in \Real^n \times [0,\infty) ~,
\end{equation}
is also spatially log-concave on $\Real^n$ for any $t \geq 0$.
\end{prop}

\begin{proof}
This follows from a standard argument, which we include for completeness.
Fix $T > 0$; we will show that $f(x,t)$ is log-concave on $\Real^n$ for any $t \in [0,T]$.
By the classical theory of parabolic PDEs (e.g. \cite[Chapter IV, Theorem 10.1]{LadyParabolicBook}), for any $0 < r < r' < R$, we have the following (spatial) interior Schauder-type estimate:
\[
\norm{f_R}_{C^{(2+\alpha ; 1+\alpha/2)}(B(r)\times[0,T])} \leq C_1 \norm{f_R(\cdot,0)}_{C^{2+\alpha}(B(r'))} + C_2 \norm{f_R}_{C^0(B(r') \times [0,T])} ~,
\]
where the constants $C_1,C_2>0$ above depend only on $n,T,\norm{\nabla U}_{C^{0,\alpha}(B(r'))},r,r',\alpha$. By the maximum principle, $\norm{f_R}_{C^0(B(r') \times [0,T])} \leq \norm{\exp(-V)}_{C^0(\Real^n)} < \infty$. And if we assume that $R \geq 1$, since $\chi$ is smooth it follows that $\norm{f_R(\cdot,0)}_{C^{2,\alpha}(B(r'))} \leq C_3 \norm{\exp(-V)}_{C^{2,\alpha}(B(r'))} < \infty$ for some constant $C_3>0$. We conclude that:
\[
\forall r > 0, \;\;\; \exists \, C_r > 0 \;\text{ such that } \; \forall R \geq r + 1, \;\;\; \norm{f_R}_{C^{(2 + \alpha; 1+\alpha/2)}(B(r)\times[0,T])} < C_r ~.
\]
It follows by Arzel\`a--Ascoli compactness that given $r > 0$, we may extract a sequence of $R_m \geq r + 1$ increasing to infinity, so that $f_{R_m}$ converges in $C^{2\times1}(B(r)\times[0,T])$. Applying a standard diagonalization argument, we conclude that there exists a sequence $\set{R_{k}}$ increasing to infinity, so that $f_{R_k}$ converges in $C^{2\times 1}_{loc}(\Real^n\times[0,T])$ to some $f_\infty \in C^{(2+\alpha; 1+\alpha/2)}_{loc}(\Real^n \times [0,T])$ (which is in addition clearly bounded). It follows that $f_\infty$ satisfies (\ref{eq:Cauchy}) on $\Real^n \times [0,T]$, so by the well-known uniqueness of this equation in the class of bounded functions, we deduce that $f_\infty \equiv f$ on $\Real^n \times [0,T]$. But $f_\infty(\cdot,t)$ is clearly log-concave for any $t \in [0,T]$, just from being the pointwise limit of the log-concave functions $f_{R_k}(\cdot,t)$. This concludes the proof.
\end{proof}

Let $V \in C^{4,\alpha}_{loc}(\Real^n)$ satisfy the assumptions of Theorem \ref{thm:main2}. If we define $V_R\in C^{4,\alpha}(B(R))$ by setting $\exp(-V_R) = \exp(-V(x)) \chi(|x|/R)$ on $B(R)$, we note that the symmetry assumptions of Theorem \ref{thm:main2} remain in tact for $V_R$ on $B(R)$. By Subsection \ref{subsec:smooth} and Proposition \ref{prop:Dirichlet}, Theorem \ref{thm:main2} consequently reduces to the following:

\begin{thm} \label{thm:max}
Let $U$ be as in Theorem \ref{thm:main1} and let $f_0 \in C^{4,\alpha}(\overline{B(R)})$ denote a positive function on $B(R)$ vanishing on $\partial B(R)$. Assume that on $B(R)$, $f_0 = \exp(-V_0)$, with $V_0$ convex and satisfying our symmetry assumptions (\ref{eq:symmetry}). Then for every $T>0$, the unique solution $f$ to the following parabolic equation on $B(R) \times [0,T]$:
\begin{equation} \label{eq:Dirichlet-PDE}
\frac{d}{dt} f = \Delta f - \scalar{\nabla f,\nabla U} ~,~  f|_{t=0} = f_0 ~ , ~  f|_{\partial B_R \times [0,T]} \equiv 0 ~,
\end{equation}
is spatially log-concave, i.e. $f = \exp(-V)$ with $V(\cdot,t)$ convex on $B(R)$ for every $t \in [0,T]$.
\end{thm}

This reduction step is similar to the one in \cite{DemangePhdThesis}, referenced to us by Cedric Villani, whom we would like to thank.

\subsection{Log-Concavity away from the boundary}

We proceed to provide a proof of Theorem \ref{thm:max}, modulo some very delicate details which are postponed to the next subsection. As in many previous works on concavity/convexity properties of solutions to elliptic and parabolic PDEs (\cite{LewisConvexRings,KorevaarClassicalPaper,Kennington2,CaffarelliSpruckConvexityProperties,KawohlBook,DiazKawohl,LeeVazquezParabolicNonlinearPDE}),
our approach is based on the maximum principle for the second derivative (or its finite difference analogue); other approaches may be found e.g. in \cite{BrascampLiebPLandLambda1,Borell1982QuasiconcavityOfBrownianMotionHittingTime,
JapaneseConvexityPreservingSingularParabolicPDE,
AlvarezLasryLions,ColesantiSalaniQuasiconcaveEnvelope,
BianGuanConstantRank} and the references therein, or in the classical book by B. Kawohl \cite{KawohlBook}.
We clarify some of the difficulties which arise in showing log-concavity in the parabolic case and which were omitted in some of these previous works. Another challenge we encounter, is that the condition our parabolic equation must satisfy, so that we can deduce the log-concavity of the solution, in fact assumes that the solution is already log-concave. Hence, arguing in the contrapositive, we must perform our analysis at precisely the \emph{first} time when things go wrong, which again requires some delicate justification.
To this end, we avoid using the usual convexity function, introduced by Korevaar \cite{KorevaarClassicalPaper} and employed by many others (see the previously mentioned references or \cite{KawohlBook,LeeVazquezParabolicNonlinearPDE,GrecoKawohl} and the references therein), and work directly with the second derivatives.

\begin{proof}[Proof of Theorem \ref{thm:max}]
By approximating $f_0$ appropriately and arguing as in Subsection \ref{subsec:smooth}, we may assume that:
\begin{equation} \label{eq:slope}
\min_{x \in \partial B(R)} |\nabla f_0|(x) > 0 ~;
\end{equation}
the only difference is that now, due to the boundary conditions, $\norm{f(\cdot,t)}_{L_1(\mu|_{B(R)})}$ will not be preserved, but rather decrease, with time. See also \cite[Lemma 6.1]{GrecoKawohl}, where a similar preliminary step was employed.

Fix $T>0$. Since $f_0 \in C^{4,\alpha}(\overline{B(R)})$ and in addition every component of $\nabla U$ is in $C^{2,\alpha}(\overline{B(R)})$, it follows from the classical Schauder theory of parabolic PDEs (e.g. \cite[Chapter IV,Theorem 10.1]{LadyParabolicBook}) that $f \in C_{loc}^{(4+\alpha;2+\alpha/2)}(B(R) \times [0,T])$ (i.e. $f \in C^{(4+\alpha;2+\alpha/2)}(K \times [0,T])$ for every compact subset $K \subset B(R)$), and also that $f \in C^{(4+\alpha;2+\alpha/2)}(\overline{B(R)} \times [\eps,T])$, for any $0<\eps < T$.
A crucial point to note is that the latter smoothness of the solution does not extend all the way to the entire boundary $\partial B(R) \times [0,T]$, since our assumption (\ref{eq:slope}) contradicts (in general) the compatibility which is usually assumed between the spatial derivatives of $f_0$ and the time derivatives of our Dirichlet conditions (see Subsection \ref{subsec:boundary}). This difficulty seems unavoidable using this approach,
and addressing it requires careful justification of subsequent steps, something which has been omitted in previous works.

It also follows from the strong maximum principle (and our initial conditions) that $f>0$ on $B(R) \times [0,T]$, and hence $V \in C^{(4+\alpha;2+\alpha/2)}_{loc}(B(R) \times [0,T])$. One immediately checks that $V$ satisfies the following non-linear parabolic PDE on $B(R) \times [0,T]$:
\[
\frac{d}{dt} V = \Delta V - \scalar{\grad V, \grad U} - \scalar{\grad V, \grad V} ~.
\]

Let $\eps > 0$ and define $\hat{V}\in C^{(4+\alpha;2+\alpha/2)}_{loc}(B(R) \times [0,T])$ as:
\[
\hat{V}(x,t) := V(x,t) + \eps \beta(t) \frac{|x|^2}{2} ~,
\]
where $\beta : [0,T] \rightarrow \Real_+$ denotes a suitable strictly positive smooth function to be determined later on. We claim that for all small enough $\eps > 0$, $\hat{V}(\cdot,t)$ must remain strictly convex for all $t \in [0,T]$, and taking the limit as $\eps \rightarrow 0$, we will conclude that $V(\cdot,t)$ is itself convex, as required.

Assume in the contrapositive that this is not so. Let $t_0 \in [0,T]$ denote the infimum over all times $t$ when $\hat{V}(\cdot,t)$ is not strictly convex, so that there exists a sequence $(x_m,t_m,e_m) \in B(R) \times (0,T] \times S^{n-1}$ converging to $(x_0 , t_0 ,e) \in \overline{B(R)} \times [0,T] \times S^{n-1}$ and satisfying $D^2_{e_m,e_m} \hat{V}(x_m,t_m) \leq 0$ (here $S^{n-1}$ denotes the unit sphere in $\Real^n$, identified with the unit sphere in the tangent spaces $T_{x_m}\Real^n$).

The most delicate part of the proof will be presented in Proposition \ref{prop:boundary} in the next subsection, where it will be shown that some further regularity estimates of $f$ up to the boundary, together with (\ref{eq:slope}) and the strict convexity of $\partial B(R)$, imply that necessarily $x_0 \notin \partial B(R)$. It follows by continuity of the second derivative of $\hat{V}$ in $B(R) \times [0,T]$ and the minimality of $t_0$ that $D^2_{e,e} \hat{V}(x_0,t_0) = 0$, and therefore $t_0 > 0$ (since at time $t=0$, $\hat{V}(\cdot,t)$ is clearly strictly convex). Moreover,
$x_0 \in B(R)$ is a local minimum point, and hence:
\begin{equation} \label{eq:max1}
D D^2_{e,e} \hat{V} (x_0,t_0) = 0 ~, ~ \Delta D^2_{e,e} \hat{V} (x_0,t_0) \geq 0 ~,~ \frac{d}{dt} D^2_{e,e}\hat{V} (x_0,t_0) \leq 0 ~,
\end{equation}
where $D$ denotes the space derivative. Since $0$ is the minimum value for the function $e \rightarrow D^2_{e,e} \hat{V}(x_0,t_0)$, it follows that it must be an eigenvalue of $D^2 \hat{V}(x_0,t_0)$, and that $e$ is a corresponding eigenvector:
\begin{equation} \label{eq:max2}
D D_e \hat{V}(x_0,t_0) = D^2 \hat{V}(x_0,t_0) e = 0 ~ , ~ \text{and hence } ~  D D_e V(x_0,t_0) = - \eps \beta(t_0) e ~.
\end{equation}
Using (\ref{eq:max1}), we must have at $(x_0,t_0)$:
\begin{equation} \label{eq:max3}
(d/dt - \Delta)(D^2_{e,e} \hat{V}) \leq 0 ~.
\end{equation}
We will show that under our assumptions on $U$ and the definition of $t_0$, the latter value must be strictly positive, obtaining the desired contradiction and concluding the proof. Indeed, at a general point $(x,t)$:
\begin{eqnarray*}
& & (d/dt - \Delta)(D^2_{e,e} \hat{V}) = D^2_{e,e}((d/dt - \Delta)(\hat{V})) \\
&=& D^2_{e,e}(\eps \beta'(t) |x|^2/2 - n \eps \beta(t) - \scalar{\grad V,\grad U} - \scalar{\grad V,\grad V}) = \eps \beta'(t) - \sscalar{D D^2_{e,e} V , DU} \\
& & - 2 \scalar{D D_e V, D D_e U} - \scalar{DV , D D^2_{e,e} U} -
2 \scalar{D D^2_{e,e} V , DV} - 2 \scalar{D D_e V, D D_e V} ~.
\end{eqnarray*}
At $(x_0,t_0)$, using $(\ref{eq:max1})$ and $(\ref{eq:max2})$, we see that:
\begin{eqnarray*}
 (d/dt - \Delta)(D^2_{e,e} \hat{V})(x_0,t_0) = \eps \beta'(t_0)  + 2 \eps \beta(t_0) D^2_{e,e} U - 2 \eps^2 \beta(t_0)^2 - \scalar{DV , D D^2_{e,e}
 U} \\
=  \eps \beta'(t_0)  + 2 \eps \beta(t_0) D^2_{e,e} U - 2 \eps^2 \beta(t_0)^2 + \eps \beta(t_0) \scalar{x, D D^2_{e,e} U} - \sscalar{D\hat{V} , D D^2_{e,e} U}  \\
\geq \eps \beta'(t_0) - (2 \eps \beta(t_0) M_2 + 2 \eps^2 \beta(t_0)^2 + \eps \beta(t_0) R M_3) - D^3 U(e,e,D \hat{V}) ~,
\end{eqnarray*}
where $M_2 := \sup_{x \in B(R), \xi \in S^{n-1}} D^2_{\xi,\xi} U(x)$ and $M_3 := \sup_{x \in B(R), \xi \in S^{n-1}} |(D^3 U)|_{x}(\xi,\xi,\frac{x}{|x|})|$.

Note that by the definitions of $t_0$ and $x_0$, $D^2_{\xi,\xi}\hat{V}(x,t_0) \geq D^2_{e,e}\hat{V}(x_0,t_0) = 0$, so $\hat{V}(\cdot,t)$ is still convex on $B(R)$ at time $t = t_0$. Also note that since $U$, $f_0$ (and $B(R)$) are all invariant under the action of $O(E_1,\ldots,E_k)$, and since the Laplace operator commutes with the entire orthogonal group, it follows easily that $f \circ G$ is also a solution to (\ref{eq:Dirichlet-PDE}) for any $G \in O(E_1,\ldots,E_k)$. The uniqueness of the solution implies that $f(\cdot,t)$ (and hence $V(\cdot,t)$ and $\hat{V}(\cdot,t)$) are also invariant under the action of this subgroup, and hence satisfy our symmetry assumptions for all $t \geq 0$. We will see in Proposition \ref{prop:geom} below that for any convex function $F : \Real^n \rightarrow \Real$ satisfying our symmetry assumptions, the condition on $U$ implies that $(D^3 U)|_{x}(\xi,\xi,D F(x)) \leq 0$ for any $x \in \Real^n$ and $\xi \in S^{n-1}$. Therefore, in order to arrive to a contradiction with (\ref{eq:max3}), it is enough to show that for small enough $\eps>0$ and an appropriate choice of $\beta$, we have:
\[
\beta'(t_0) - (2 \beta(t_0) M_2 + 2 \eps \beta(t_0)^2 + \beta(t_0) R M_3) > 0 ~.
\]
Indeed, this is satisfied on $[0,T]$ by setting $\beta(t) := \exp((2M_2 + R M_3 + 1) t)$ and letting $\eps < 1/(2\beta(T))$. This completes the contradiction and concludes the proof, modulo Propositions \ref{prop:geom} and \ref{prop:boundary} below.
\end{proof}

We conclude this subsection with the proof of the following proposition, which is the only place where we use our structural assumptions on $U$ and $V$. In fact, the assumption that $U$ is convex may be omitted in all instances below (see Section \ref{sec:comparing} for more on this).

\begin{prop} \label{prop:geom}
If $U$ and $V$ satisfy the assumptions of Theorem \ref{thm:main1} then:
\[
(D^3 U)|_{x}(\xi,\xi,\grad V(x)) \leq 0 \;\;\;  \forall x \in \Real^n ~ \forall \xi \in S^{n-1} ~.
\]
\end{prop}

The proposition follows immediately from the following two lemmata, which we formulate separately for later use:

\begin{lem} \label{lem:geom1}
Let $U$ satisfy the assumptions of Theorem \ref{thm:main1}. Then $(D^3 U)|_{x}(\xi,\xi,\theta) \leq 0$, for any $x \in \Real^n$, $\xi \in S^{n-1}$ and $\theta \in S^{n-1}$ such that:
\begin{equation} \label{eq:theta-cond}
\forall i=1,\ldots,k \;\;\; \exists a_i \geq 0 \;\;\; \text{so that} \;\;\; Proj_{E_i} \theta = a_i Proj_{E_i} x ~.
\end{equation}
\end{lem}

\begin{lem} \label{lem:geom2}
Let $V$ satisfy the assumptions of Theorem \ref{thm:main1}. Then for any $x \in \Real^n$, $\theta = \nabla V(x)$ satisfies (\ref{eq:theta-cond}).
\end{lem}

\begin{proof}[Proof of Lemma \ref{lem:geom1}]
Let $\varrho_i : E_i \rightarrow \Real$ be given by $\varrho_i(x) = \rho_i(|x|)$, $i=1,\ldots,k$.
Taking the third derivative of $U$, the quadratic term in (\ref{eq:U}) disappears and we are left with:
\[
 \left . (D^3 U) \right |_{x}(\xi,\xi,\theta) = \sum_{i=1}^k \left . (D^3_{E_i} \varrho_i) \right |_{Proj_{E_i} x}(Proj_{E_i} \xi,Proj_{E_i} \xi,Proj_{E_i} \theta) ~.
\]
Let us show that each summand is non-positive. Denote:
\[
 x_i := Proj_{E_i} x ~,~ \xi_i := Proj_{E_i} \xi ~,~ \xi_i^r := Proj_{x_i} \xi_i ~,~ \xi_i^t := Proj_{x_i^\perp} \xi_i  ~,~ \theta_i := Proj_{E_i} \theta ~.
\]
If $x_i = 0$ then $\theta_i = 0$ and hence the $i$-th summand is also $0$, so we may assume that $x_i \neq 0$.
Using (\ref{eq:theta-cond}), an elementary calculation yields:
\[
 \left . (D^3_{E_i} \varrho_i) \right |_{x_i} (\xi_i,\xi_i,\theta_i) = \brac{ \rho_i'''(|x_i|) |\xi_i^r|^2 + \brac{\rho_i''(|x_i|) - \frac{\rho_i'(|x_i|)}{|x_i|}} \frac{|\xi_i^t|^2}{|x_i|} } a_i |x_i| ~.
\]
Since $t \mapsto \rho_i(|t|)$ is a $C^3$ function, we see that $\rho_i'(0) = 0$. Since $\rho'''_i \leq 0$ on $\Real_+$, meaning that $\rho'_i$ is concave there, we deduce that also $\rho_i''(t) \leq (\rho_i'(t)-\rho_i'(0))/t = \rho_i'(t)/t$ for all $t > 0$. This implies that the term in brackets on the rand-hand side above is non-positive, and since $a_i \geq 0$, the entire expression is non-positive as well, as claimed.
\end{proof}

\begin{proof}[Proof of Lemma \ref{lem:geom2}]
Denote as usual $x_i = Proj_{E_i} x$, $i=0,1,\ldots,k$. Let us verify (\ref{eq:theta-cond}) for each $i=1,\ldots,k$.
It is easy to see that the symmetries of $V$ ensure that $D_i V(x) := Proj_{E_i} \nabla V(x)$ lies in the one-dimensional subspace spanned by $x_i$.
Hence if $x_i = 0$, then $D_i V(x) =0$ and (\ref{eq:theta-cond}) is satisfied trivially for that $i$, so we may assume otherwise. Denoting:
\[
D_i V(x) =: D_i^r V(x) \frac{x_i}{|x_i|} ~,
\]
it remains to verify that $D_i^r V(x) \geq 0$ when $x_i \neq 0$.
The symmetries of $V$ together with its convexity together imply that the following (convex) slice of $V$'s sub-level set at $x$:
\[
A(x):=\set{z \in E_0^\perp ;  V(x_0 + z) \leq V(x)} ~,
\]
contains the product set $B_{E_1}(|x_1|) \times \ldots \times B_{E_k}(|x_k|)$, where $B_{E_i}(r)$ denotes the Euclidean ball of radius $r$ in $E_i$. Geometrically, this means that the latter product set lies entirely on one side of the tangent plane to $A(x)$ at $Proj_{E_0^\perp} x$, or more precisely, that:
\[
 \scalar{Proj_{E_0^\perp} \grad V(x) , R(x) - x} \leq 0 \;\;\; \forall R \in O(E_1,\ldots,E_k) ~.
\]
Recalling that $Proj_{E_0^\perp} \grad V(x) = \sum_{i=1}^k D_i^r V(x) x_i$ and choosing $R_i \in O(E_1,\ldots,E_k)$ to be the reflection in $E_i$, defined by $R_i(x) = x - 2 x_i$, we conclude that:
\[
 D_i^r V(x) |x_i|^2 \geq 0 \;\;\; \forall i=1,\ldots,k ~.
\]
Since we assumed that $x_i \neq 0$, it follows that $D_i^r V(x) \geq 0$, as required.
\end{proof}

\subsection{Log-Concavity near the boundary} \label{subsec:boundary}

To complete the proof of Theorem \ref{thm:max}, we must show that $x_0 \notin \partial B(R)$.
Recalling the definition of $x_0$, this clearly follows from:
\begin{prop}\label{prop:boundary}
$D^2 V(x,t) \geq 0$ in a neighborhood of $\partial B(R) \times [0,T]$.
\end{prop}
The proof of Proposition \ref{prop:boundary} will be given at the end of this section, but first we explain the subtle regularity issue one is required to address here. Recall that the classical theory guarantees that under the assumptions of Theorem \ref{thm:max}, $f \in C_{loc}^{(4+\alpha;2+\alpha/2)}(B(R) \times [0,T])$ (i.e. $f \in C^{(4+\alpha;2+\alpha/2)}(K \times [0,T])$ for every compact subset $K \subset B(R)$), and also that $f \in C^{(4+\alpha;2+\alpha/2)}(\overline{B(R)} \times [\eps,T])$, for any $0<\eps < T$. However, the latter smoothness does not extend all the way to the ``corner'' $\partial B(R) \times \set{0}$, since in general we cannot guarantee the necessary and sufficient compatibility conditions:
\begin{equation} \label{eq:compat}
L^i(f_0)|_{\partial B(R)} \equiv 0  \;\;\;\;\; i=1,2
\end{equation}
(here $L^i$ denotes the iterated application of the operator $L$). This prevents a straightforward application of standard arguments for deducing Proposition \ref{prop:boundary}, and so we consequently need to obtain some delicate regularity estimates \emph{up to the boundary} for the solution $f$ to (\ref{eq:Dirichlet-PDE}), which are given in Proposition~\ref{prop:regularity} below.

\medskip

To outline the proof and properly motivate Proposition~\ref{prop:regularity}, observe that:
\[
D^2 V = -D^2 \log f = \frac{ - f D^2 f + \nabla f \otimes \nabla f}{f^2 } ~.
\]
Using Hopf's maximum principle and continuity of $\nabla f$ (see Proposition~\ref{prop:regularity} (1)) we see below that $\nabla f$ is bounded uniformly away from zero near $\partial B(R) \times [0, T]$. Therefore, the term $\nabla f \otimes \nabla f$ is uniformly positive definite when restricted to the normal direction (relative to $\partial B(R)$). In addition, the gradient bound implies that $f$ decays linearly to $0$ near the boundary, and one can show that $-fD^2 f$ decays uniformly to zero near $\partial B(R) \times [0,T]$ (see Proposition~\ref{prop:regularity} (2)). It follows that $D^2 V$ restricted to the normal direction is uniformly positive definite near $\partial B(R) \times [0,T]$.  On the other hand, since $\partial B(R)$ is the zero level set of $f$, the uniform convexity of $\partial B(R)$ and the uniform lower bound on $|\nabla f|$ together imply that $-D^2 f$ restricted to the tangential directions is uniformly positive definite along $\partial B(R) \times [0, T]$. Since the tangential second derivatives of $f$ are uniformly continuous (see Proposition~\ref{prop:regularity} (3)), it follows that $D^2 V$ restricted to the tangential directions is uniformly positive definite in a neighborhood of $\partial B(R) \times [0,T]$. Mixed derivatives are controlled similarly.

We now proceed with providing the precise details. We begin with:

\begin{prop} \label{prop:regularity}
Under the assumptions of Theorem \ref{thm:max}:
\begin{enumerate}
\item $f \in C^{(1+\beta; (1+\beta)/2)}(\overline{B(R)} \times [0,T])$ for all $\beta \in (0,1)$.
\item For any $\eps > 0$ there exists a $C_\eps > 0$ so that for any $\lambda \in (0,R)$:
\begin{equation} \label{eq:weak-C^2}
\sup_{t \in [0,T]} \norm{f(\cdot,t)}_{C^{2}(\overline{B(R-\lambda)})} \leq \frac{C_\eps}{\lambda^{\eps}} ~.
\end{equation}
\item If $n \geq 2$, the spatial derivatives of $f$ in the non-radial directions are $C^{1,\delta}(\overline{B(R)})$ uniformly in  $t \in [0,T]$, for any $\delta \in (0,1)$. In other words, for any $\delta \in (0,1)$ there exists a finite constant $C_\delta>0$, so that for any smooth unit vector field $\xi$ on $\overline{B(R)}$
 such that $\scalar{\xi(x),x}\equiv 0$:
\[
\sup_{t \in [0,T]} \norm{\scalar{\grad f(\cdot,t),\xi(\cdot)}}_{C^{1,\delta}(\overline{B(R)})} \leq C_\delta ~;
\]
(in fact, we actually have $\norm{\scalar{\grad f,\xi}}_{C^{(1+\delta;(1+\delta)/2)}(\overline{B(R)}\times[0,T])} \leq C_\delta$ ).
\end{enumerate}
\end{prop}

\begin{rem}
We were informed by Ki-Ahm Lee and Vladimir Maz'ya that it should actually be true that:
\[
\sup_{t \in [0,T]} \norm{f(\cdot,t)}_{C^{1,1}(\overline{B(R)})} < \infty ~,
\]
but we will not insist on this here since the easier weaker estimate (\ref{eq:weak-C^2}) suffices for our purposes.
\end{rem}

\begin{proof}[Proof of Proposition \ref{prop:regularity}]

1. The first assertion follows from standard regularity theory. Even if the compatibility conditions (\ref{eq:compat}) do not necessarily hold, it follows by \cite[Theorem 5.1.11 (ii)]{LunardiBook} that when $f_0 \in C^{1,\beta}(\overline{B(R)})$ for some $\beta \in (0,1)$ and $f_0|_{\partial B(R)} \equiv 0$, then:
\begin{equation} \label{eq:reg-1}
f \in C^{(1+\beta;(1+\beta)/2)}(\overline{B(R)} \times [0,T]) ~.
\end{equation}
Alternatively, one may employ the Sobolev regularity theory for parabolic PDEs (e.g. \cite[Chapter IV, Theorem 9.1 and subsequent Corollary]{LadyParabolicBook}), which ensures that $f \in W_p^{(2;1)}(\overline{B(R)} \times [0,T])$ for all $p \in (1,\infty)$. Consequently, (\ref{eq:reg-1}) follows by a variant of Morrey's embedding theorem (e.g. \cite[Chapter II, Lemma 3.3]{LadyParabolicBook}).

\bigskip

\noindent 2. This may be deduced from \cite[Theorem 5.15]{LiebermanBook} by considering weighted H\"{o}lder spaces. To avoid these, one may proceed as follows. Applying a standard Schauder-type interior estimate, if $f_0 \in C^{2,\gamma}(\overline{B(R)})$ and each component of $\grad U$ is in $C^{0,\gamma}(\overline{B(R)})$, one checks (see e.g. \cite[p. 355]{LadyParabolicBook}) that:
\begin{equation} \label{eq:reg-2}
 \norm{f}_{C^{(2+\gamma;1+\gamma/2)}(\overline{B(R-\lambda)}\times[0,T])} \leq \frac{C_\gamma}{\lambda^{2+\gamma}} \;\;\; \forall \lambda \in (0,R) ~.
\end{equation}
Combining (\ref{eq:reg-1}) and (\ref{eq:reg-2}), we deduce under the assumptions of Theorem \ref{thm:max}, that for all $\lambda \in (0,R)$:
\begin{eqnarray*}
\sup_{t \in [0,T]} \norm{f(\cdot,t)}_{C^{1,\beta}(\overline{B(R-\lambda)})} \leq B_\beta \;\;\; \forall \beta \in (0,1) ~ ; \\
\sup_{t \in [0,T]} \norm{f(\cdot,t)}_{C^{2,\gamma}(\overline{B(R-\lambda)})} \leq \frac{C_\gamma}{\lambda^{2+\gamma}} \;\;\; \forall \gamma \in (0,1) ~.
\end{eqnarray*}
Since $\partial B(R-\lambda)$ is uniformly smooth for all $\lambda \in (0,R/2)$, one can use interpolation in the spaces of H\"{o}lder differentiable functions (see Lunardi \cite[Corollary 1.2.19,1.2.7]{LunardiBook}), and obtain for any $\eta \in (0,\gamma)$ and $\lambda$ in this range:
\[
\sup_{t \in [0,T]} \norm{f(\cdot,t)}_{C^{2,\eta}(\overline{B(R-\lambda)})}
\leq A_{2+\gamma,2+\eta,1-\beta} B_{\beta}^{\frac{\gamma-\eta}{\gamma+1-\beta}} C_{\gamma}^{\frac{1-\beta+\eta}{\gamma+1-\beta}} \lambda^{-\frac{(2+\gamma)(1-\beta+\eta)}{\gamma+1-\beta}} ~.
\]
By modifying the constants above, the bound remains valid for all $\lambda \in (0,R)$.
Choosing $\eta > 0$ and $1-\beta > 0$ very small, the second part of Proposition \ref{prop:regularity} follows.

\bigskip

\noindent 3. This part is obtained by first flattening the boundary $\partial B(R)$ near a point, and then
 applying the standard parabolic regularity theory to the resulting PDE for $D_\tau f$,  where $\tau$ denotes a vector parallel to the flattened boundary. This procedure is standard, and the details are provided for the reader's convenience.

Let us fix an orthogonal basis $e_1,\ldots,e_n$ of $(\Real^n,|\cdot|)$ and a direction $\xi_0 \in S^{n-1}$.
Let $T : \overline{B(R)} \rightarrow \overline{\Omega}$ denote a smooth diffeomorphism so that $T$ coincides with the usual
Cartesian-to-polar change of coordinates on the half-annulus $A_+ := B(R) \setminus \overline{B(R/2)}) \cap \set{x \in \Real^n ; \scalar{x,\xi_0} > 0}$. Now consider the PDE satisfied by $g = f \circ T^{-1}$ on $\Omega$. Since both $T$ and $T^{-1}$ are smooth and in particular Lipschitz, it is easy to check that $g$ satisfies a uniformly parabolic PDE on $\Omega \times [0,T]$ of the form:
\begin{equation} \label{eq:polar-PDE}
\frac{d}{dt} g = \sum_{i,j} a_{i,j} D^2_{i,j} g + \sum_i b_i D_i g ~,
\end{equation}
where $a_{i,j} = a_{i,j}(y)$ is a uniformly positive-definite smooth matrix and $b_i = b_i(y)$ have the same smoothness as $\grad U$, i.e. $b_i \in C^{2,\alpha}(\overline{\Omega})$. Moreover, since in polar-coordinates:
\[
\Delta = r^{-n+1} \frac{\partial}{\partial r} (r^{n-1} \frac{\partial}{\partial r}) + \frac{1}{r^2} \Delta_{S^{n-1}} ~,
\]
we see that on $T(A_+)$, if we use the natural basis $y = (\theta_1,\ldots,\theta_{n-1},r)$ to write (\ref{eq:polar-PDE}), we actually have:
\begin{equation} \label{eq:polar-matrix}
a_{i,j}(\theta_1,\ldots,\theta_{n-1},r) = \begin{cases} \delta_{i,j}  & i=n \\ \frac{\delta_{i,j}}{r^2} & i=1,\ldots,n-1 \end{cases} ~.
\end{equation}
Finally, since $T$ is a diffeomorphism, $T(\partial B(R)) = \partial \Omega$, and hence the boundary conditions are given by:
\[
g|_{t=0} = g_0 := f_0 \circ T^{-1}  \;\;\;, \;\;\; g|_{\partial \Omega \times [0,T]} \equiv 0 ~.
\]

The usual regularity theory ensures that $g \in C_{loc}^{(4+\alpha;2+\alpha/2)}(\Omega \times [0,T])$, and as in the first part, it follows that:
\begin{equation} \label{eq:g-reg}
g \in C^{(1+\delta;(1+\delta)/2)}(\overline{\Omega} \times [0,T]) \;\;\; \forall \delta \in (0,1) ~.
\end{equation}
Now take the spatial derivative of (\ref{eq:polar-PDE}) in a direction $\tau \in \text{span}(e_1,\ldots,e_{n-1})$. Denoting $g_\tau := D_\tau g$, we obtain
that in $\Omega \times [0,T]$:
\[
\frac{d}{dt} g_\tau = \sum_{i,j} a_{i,j} D^2_{i,j} g_\tau + \sum_{i,j} D_\tau a_{i,j} D^2_{i,j} g + \sum_i b_i D_i g_\tau +  \sum_i D_\tau b_i D_i g ~.
\]
By (\ref{eq:g-reg}), the fourth term on the right hand side, which we denote by $h$, is in $C(\overline{\Omega}\times[0,T])$ (and in fact better). The second term above contains mixed second derivatives of $g$, but fortunately in $T(A_+)$, the matrix $a_{i,j}(y)$ is given by (\ref{eq:polar-matrix}), and hence $D_\tau a_{i,j}(y) = 0$. We conclude that in $T(A_+) \times [0,T]$, $g_\tau$ satisfies the following uniformly parabolic PDE:
\begin{equation} \label{eq:g-tau-PDE}
\frac{d}{dt} g_\tau = \sum_{i,j} a_{i,j}(y) D^2_{i,j} g_\tau + \sum_i b_i(y) D_i g_\tau + h(y,t) ~,
\end{equation}
and that:
\[
g_\tau|_{t=0} = D_\tau g_0 \;\;\; , \;\;\; g_\tau|_{(\partial T(A_+) \cap \partial \Omega) \times [0,T]} \equiv 0 ~.
\]

Employing the standard regularity theory, it follows as in the first part that $g_\tau \in C^{(1+\delta;(1+\delta)/2)}(\overline{\Theta}\times[0,T])$ for any $\delta \in (0,1)$ and open subset $\Theta \subset \Omega$ with smooth boundary, which is in addition bounded away from $\partial T(A_+) \setminus \partial \Omega$.
Recalling that $g = f \circ T^{-1}$ and that $T$ is a polar change-of-coordinates on $T(A_+)$, the third assertion of the proposition follows on $(B(R\xi_0,a) \cap \overline{B(R)}) \times [0,T]$ for some small enough $a>0$ (here $B(z,a)$ denotes the ball of radius $a$ centered at $z$).
By following the bounds obtained in the proof, one may check that these do not depend on the choice of $\xi_0$ or the non-radial direction $\tau$. By compactness (or using the fact that actually $a>0$ does not depend on $\xi_0$), the assertion follows on a uniform neighborhood of $\partial B(R) \times [0,T]$, and the classical theory takes care of the interior regularity. This completes the proof.
\end{proof}

\begin{proof}[Proof of Proposition \ref{prop:boundary}]
Recall that by the classical theory, $f(\cdot,t) \in C^{4,\alpha}(\overline{B(R)})$ for every $t \in [0,T]$. The second fundamental form of a spatial level set $M$ of $f$ at a point $(x,t)$ with $\nabla f(x,t) \neq 0$, i.e. $M = M_{v,t} := \set{z \in \overline{B(R)} ; f(z,t) = v}$ where $v = f(x,t)$, is given by:
\[
II_{M}(x) = - \left . D \frac{\nabla f}{|\nabla f|} \right |_{T_x M}  = - \left . \frac{D^2 f}{|\nabla f|} \brac{Id - \frac{\nabla f}{|\nabla f|} \otimes \frac{\nabla f}{|\nabla f|}} \right |_{T_x M} =
- \left . \frac{D^2 f}{|\nabla f|} \right |_{T_x M} ~.
\]
Since we assumed in (\ref{eq:slope}) that $|\nabla f_0 |  > 0$ on $\partial B(R)$ and since $\nabla f$ is (uniformly) continuous on $\overline{B(R)}\times[0,T]$ by Proposition \ref{prop:regularity} (1), it follows that there exists some $T_0 > 0$ so that $|\nabla f| \geq c' > 0$ on all of $\partial B(R) \times [0,T_0]$. By the strong maximal principle and Hopf's lemma in the parabolic setting (see e.g. \cite[Chapter 2, Theorem 14]{FriedmanPDEBook}), $|\nabla f| \geq c'' > 0$ on all of $\partial B(R) \times [T_0,T]$, and by the uniform continuity of $\nabla f$ we conclude that there exists $R' \in (0,R)$ and $c,C >0$ so that:
\[
0 < c \leq -\scalar{\nabla f(x,t),\frac{x}{|x|}} \leq |\nabla f(x,t)| \leq C , \;\;\; \forall |x| \in [R',R] \;\;\; \forall t \in [0,T] ~,
\]
and hence:
\begin{equation} \label{eq:good-slope2}
c (R - |x|) \leq f(x,t) \leq C (R- |x|), \;\;\; \forall |x| \in [R',R] \;\;\; \forall t \in [0,T] ~.
\end{equation}
Since the level set $M_{0,t}$ coincides with $\partial B(R)$ for all $t \in [0,T]$ ($f>0$ in $B(R)\times [0,T]$ by the strong maximum principle), it follows that:
\[
- \left . \frac{D^2 f}{|\nabla f|} \right |_{x^\perp} = \frac{1}{R} \left . Id \right|_{x^\perp} \;\;\; \forall (x,t) \in \partial B(R) \times [0,T] ~,
\]
where $x^\perp$ is identified with $T_x \partial B(R)$.
By Proposition \ref{prop:regularity} (3), the second spatial derivatives of $f$ involving a non-radial direction are (uniformly) continuous on $\overline{B(R)} \times [0,T]$, and so we deduce that there exists some $R'' \in [R',R)$ so that:
\[
-D^2_{\tau,\tau} f(x,t) \geq \frac{c}{2 R} \text{ and } -D^2_{\tau,\frac{x}{|x|}} f(x,t) \geq -B
\;\;\;\; \forall |x| \in [R'',R] \;\;\; \forall t \in [0,T] \;\;\; \forall \tau \in S^{n-1} \cap x^\perp ~.
\]
where:
\[
B := \max\set{ \abs{D^2_{\tau,\frac{x}{|x|}} f(x,t)} \; ; \; x \in \overline{B(R)} , \tau \in S^{n-1} \cap x^\perp , t \in [0,T] } < \infty ~.
\]
Since the tangential derivatives of $f$ vanish on $\partial B(R)$, it also follows that:
\[
|\scalar{\grad f(x,t), \tau}| \leq B (R - |x|) \;\;\;\; \forall |x| \in [R'',R] \;\;\; \forall t \in [0,T] \;\;\; \forall \tau \in S^{n-1} \cap x^\perp ~.
\]
Lastly, fixing $\eps \in (0,1)$, it follows by Proposition \ref{prop:regularity} (2) and (\ref{eq:good-slope2}) that:
\[
-f(x,t) D^2_{\xi,\xi} f (x,t) \geq - C C_\eps (R - |x|)^{1-\eps} \;\;\;\; \forall |x| \in [R'',R] \;\;\; \forall \xi \in S^{n-1} \;\;\; \forall t \in [0,T] ~.
\]

We are now ready to bound $D^2 V$, using:
\[
D^2 V = -D^2 \log f = \frac{ - f D^2 f + \nabla f \otimes \nabla f}{f^2 } ~.
\]
Given $x$ with $|x| \in [R'',R]$ and a direction $\xi \in S^{n-1}$, write $\xi = \cos(\theta) \tau + \sin(\theta) \rho$, where $\rho = x/|x|$ and $\scalar{\tau,\rho} = 0$. For the purpose below, we can assume without loss of generality that $\theta \in [0,\pi/2]$. At the point $(x,t)$, denoting in addition $d = R - |x|$,
we have by all the estimates above:
\begin{eqnarray*}
f^2 D^2_{\xi,\xi} V  & = & \cos^2(\theta) (- f D^2_{\tau,\tau} f + \scalar{\grad f,\tau}^2 ) + \sin^2(\theta) (- f D^2_{\rho,\rho} f + \scalar{\grad f,\rho}^2 ) \\
& + & 2 \sin(\theta) \cos(\theta) ( - f D^2_{\tau,\rho} f + \scalar{\grad f,\tau} \scalar{\grad f,\rho}) \\
& \geq & \cos^2(\theta) c d \frac{c}{2R} + \sin^2(\theta) ( - C C_\eps d^{1-\eps} + c^2) + 2 \cos(\theta) \sin(\theta) (-C B d  - C B d) ~.
\end{eqnarray*}
We see that if $d \in [0,d_0]$ for some small enough $d_0 \in (0,R-R'']$, we have for some $p,q,r,p',q' > 0$:
\begin{eqnarray*}
f^2 D^2_{\xi,\xi} V & \geq & \cos^2(\theta) p d + \sin^2(\theta) q - 2 \cos(\theta)\sin(\theta) r d \\
& \geq & \cos^2(\theta) \frac{p}{2} d  + \sin^2(\theta) \brac{q - \frac{2 r^2}{p} d} \geq \cos^2(\theta) p' d + \sin^2(\theta) q' ~,
\end{eqnarray*}
and so when $d \in (0,d_0]$:
\[
D^2_{\xi,\xi} V \geq \frac{\cos^2(\theta) p' d + \sin^2(\theta) q'}{C^2 d^2} > 0 ~;
\]
(indeed, this behaviour as a function of $\theta,d$ is the best one can expect).
We conclude that $D^2_{\xi,\xi} V(x,t) \geq 0$ (and in fact, tends to $+\infty$ uniformly in $d$) for all $|x| \in [R-d_0,R]$, $t \in [0,T]$ and $\xi \in S^{n-1}$. The proof is complete.
\end{proof}

\section{Tying up loose ends} \label{sec3}

In this section, we provide a complete justification of the proof of Theorem \ref{thm:main1}, described in the Introduction.  We proceed with the same notations used there. The main technical points which we address in this section are showing that the flow map $S_t$ is globally well-defined on $\mathbb{R}^n$ (see Lemma~\ref{lem:hessian of V bound} and its preceding discussion), that the pushed-forward measure $\nu_t := (S_{t})_* \nu= P_t^U(\exp(-V)) \mu$ converges to  $\mu$ (see Lemma~\ref{lem:convergence}), and that the inverse map  $T_t = S_t^{-1}$ converges (to a contracting map) as $t \to \infty$ (see Lemma~\ref{lem:approx})
.

Let $U,V$ be as in Theorem \ref{thm:main1}. We assume further that $V$ is sufficiently smooth (e.g. $V \in C^{4,\alpha}(\Real^n)$ is more than enough), and that:
\begin{equation} \label{eq:sec3-assumptions}
\norm{\nabla V}_{C^{1,\alpha}(\Real^n)} < \infty \text{ and } \norm{D^3 U}_{L_\infty} < \infty ~.
\end{equation}
We will see how to obtain the general case at the very end of this section.

First, since $\exp(-V) \in C^{4,\alpha}(\Real^n)$ and $U \in C^{3,\alpha}_{loc}(\Real^n)$, the classical regularity theory of parabolic PDEs (e.g. \cite{LadyParabolicBook}) ensures that $f(x,t):=P_t^U(\exp(-V))(x)$, as the unique (bounded) solution to the Cauchy problem:
\begin{equation} \label{eq:sec3-PDE}
 \frac{d}{dt} f = L f ~,~ f|_{t=0} = \exp(-V) ~,
\end{equation}
is $C^{(4+\alpha ; 2+\alpha/2)}_{loc}(\Real^n \times [0,\infty))$, and the strong maximum principle ensures that $f(x,t)$ is strictly positive for all $t \in [0,\infty)$.
Consequently, the advection field $W_t := -\nabla \log P_t^U(\exp(-V))$ is $C^{(3+\alpha;(3+\alpha)/2)}_{loc}(\Real^n \times [0,\infty))$. In particular, the maps $S_t$ defined by:
\begin{equation} \label{eq:ODE}
 \frac{d}{dt} S_t(x) = W_t(S_t(x)) ~,~ S_0 = Id ~,
\end{equation}
are indeed \emph{locally} well-defined as a solution to a flow along a locally Lipschitz vector field (e.g. \cite[Proposition 1.56]{GHLBookEdition2}): for any compact subset $C \subset \Real^n$, there exists $t(C) > 0$, so that (\ref{eq:ODE}) has a solution for any $(x,t) \in C \times [0,t(C))$. To ensure that the maps $S_t$ are globally well-defined, it is enough to show that for any $T > 0$, $W_t(x)$ is \emph{globally} spatially Lipschitz for all $t \in [0,T]$, i.e. $|D W_t(x)| < C(T)$ for all $(x,t) \in \Real^n \times [0,T]$:

\begin{lem}\label{lem:hessian of V bound}
Assuming (\ref{eq:sec3-assumptions}), for all $T>0$, $D^2 \log P_t^U(\exp(-V))(x)$ is uniformly bounded in $\Real^n \times [0,T]$.
\end{lem}

\begin{proof}
We denote by abuse of notation $V = V(x,t) = - \log P_t^U(\exp(-V))(x)$ and $V_t = V(\cdot,t)$. Since $D^2 V \geq 0$ by Theorem \ref{thm:main2}, it suffices to show a uniform bound on $Z = \Delta V$. Recall from Section \ref{sec2} that $V$ satisfies:
\begin{equation} \label{eq:V-PDE}
\frac{d}{dt} V = \Delta V - \scalar{\nabla V,\nabla V} - \scalar{\nabla V,\nabla U} ~,~ V|_{t=0} = V_0 ~.
\end{equation}
A direct calculation gives:
\begin{align*}
 \frac{d}{dt} Z  = & \Delta Z - 2 \scalar{\nabla Z , \nabla V} - \scalar{\nabla Z,\nabla U}\\ &  - 2 tr((D^2 V)^* D^2 V) - 2 tr((D^2 V)^* D^2 U) - \scalar{\nabla \Delta U,\nabla V} ~.
\end{align*}
Recalling that $D^2 U \geq 0$ and $D^2 V \geq 0$, we conclude that:
\begin{equation} \label{eq:Z-PDE}
\frac{d}{dt} Z \leq \Delta Z - 2 \scalar{\nabla Z , \nabla V} - \scalar{\nabla Z,\nabla U} - \scalar{\nabla \Delta U,\nabla V} ~.
\end{equation}

To apply the maximum principle to (\ref{eq:Z-PDE}), we need to control the zeroth order (right-most) term. To this end, we claim that:
\begin{equation} \label{eq:DV-bound}
\norm{\nabla V_t}_{L_\infty} \leq \norm{\nabla V_0}_{L_\infty}  \;\;\; \forall t \geq 0 ~.
\end{equation}
This follows e.g. by using the pointwise estimate of Bakry and \'Emery, refined by Bakry \cite[Proposition 1]{BakryRieszPartII}, which when $U$ is convex yields $|\nabla P_t^U(f)| \leq P_t^U(|\nabla f|)$. Together with the maximum principle, this indeed implies that:
\[
|\nabla V_t(x)| = \frac{|\nabla P_t^U(\exp(-V_0))(x)|}{P_t^U(\exp(-V_0))(x)} \leq \frac{P_t^U(|\nabla V_0|\exp(-V_0))(x)}{P_t^U(\exp(-V_0))(x)} \leq \norm{\nabla V_0}_{L_\infty} ~.
\]

Now applying formally the maximum principle to (\ref{eq:Z-PDE}), using (\ref{eq:DV-bound}) and recalling the definition of $Z$, we obtain:
\[
\norm{\Delta V_t}_{L_\infty} \leq \norm{\Delta V_0}_{L_\infty} + t\, n \norm{D^3 U}_{L_\infty} \norm{D V_0}_{L_\infty}  ~.
\]
The assumption (\ref{eq:sec3-assumptions}) ensures (in particular) that all terms above are bounded, and hence $\Delta V_t$ is uniformly bounded on $[0,T]$ and it seems that we are done.

However, there is a technical issue here: to appeal to the maximum principle on the unbounded domain $\Real^n$, we have to {\em a-priori} verify that $\Delta V_t (x)$ does not grow spatially faster than $\exp(C |x|^2)$ for some $C>0$, uniformly in $t \in [0,T]$ (see e.g. \cite[Chapter 2, Theorem 9]{FriedmanPDEBook}). The rest of the proof is dedicated to verifying this a-priori growth rate.

\bigskip

First, observe that $V$ grows spatially at most linearly, uniformly in $t \in [0,T]$. To see this without eluding to compactness, denote by $m$ the minimum of $V_0$, and hence (by the maximum principle) of $V(\cdot,t)$ for any $t \geq 0$. Fix $C>0$ and let $r>0$ be so that:
\[
\exp(-C) \mu(B(r)) + \exp(-m) (1 - \mu(B(r))) < 1 ~.
\]
It follows since $\int \exp(-V(x,t)) d\mu(x) = 1$ for any $t \geq 0$, that for any such $t$ there exists $x_0(t) \in B(r)$ so that $V(x_0(t),t) \leq C$.
Consequently, (\ref{eq:DV-bound}) implies that $V(x,t) \leq \norm{\nabla V_0}_{L_\infty} |x -x_0(t)| + V(x_0(t),t) \leq \norm{\nabla V_0}_{L_\infty} (|x|+r) + C$.

Now write (\ref{eq:V-PDE}) as:
\[
 \frac{d}{dt} V - \Delta V = h ~,~ V|_{t=0} = V_0 ~,
\]
where $-h = \scalar{\nabla V,\nabla V} + \scalar{\nabla V,\nabla U}$. By the assumptions of Theorem \ref{thm:main1}, $|\nabla U|(x)$ grows at most linearly in $|x|$, and together with (\ref{eq:DV-bound}), it follows that $h$ too grows spatially at most linearly. Consequently, applying an interior regularity estimate (e.g.
applying the estimate \cite[Chapter IV, (10.2)]{LadyParabolicBook} for the Sobolev space $W_p^{(2;1)}$ with $p$ arbitrarily large, followed by a variant of Morrey's embedding theorem as in the Corollary after \cite[Chapter IV,Theorem 9.1]{LadyParabolicBook}), it follows that:
\begin{eqnarray*}
&      & \norm{V}_{C^{(1+\alpha;(1+\alpha)/2)}(B(R)\times[0,T])} \\
& \leq & C(n,T,\alpha) ( \norm{h}_{C^0(B(R')\times[0,T])} + \norm{V_0}_{C^2(B(R'))} + \norm{V}_{C^{0}(B(R') \times [0,T])}) ~,
\end{eqnarray*}
for any $\alpha \in (0,1)$ and $1 \leq R \leq R'-1$. Since $\norm{\nabla V_0}_{C^{1+\alpha}(\Real^n)}$ is assumed bounded in (\ref{eq:sec3-assumptions}), and as explained above, $V_0$, $V$ and $h$ grow spatially at most linearly, it follows that so does $\norm{V}_{C^{(1+\alpha;(1+\alpha)/2)}(B(R)\times[0,T])}$.

Using this and arguing as above, we verify that $\norm{h}_{C^{(\alpha;\alpha/2)}(B(R)\times[0,T])}$ grows at most quadratically in $R$. Applying the interior Schauder estimate again (e.g. \cite[Chapter IV, Theorem 10.1]{LadyParabolicBook}), it follows that:
\begin{eqnarray*}
& & \norm{V}_{C^{(2+\alpha;1+\alpha/2)}(B(R)\times[0,T])} \\
& \leq &  C(n,T,\alpha) (
\norm{h}_{C^{(\alpha;\alpha/2)}(B(R')\times[0,T])} +
\norm{V_0}_{C^{2+\alpha}(B(R'))} +
\norm{V}_{C^0(B(R')\times[0,T])}
) ~,
\end{eqnarray*}
for any $1 \leq R \leq R'-1$. Using (\ref{eq:sec3-assumptions}) again, we conclude that $D^2 V_t$ \emph{a-priori} spatially grows at most polynomially, uniformly in $t \in [0,T]$, thereby concluding the proof.
\end{proof}

We conclude from Lemma~\ref{lem:hessian of V bound} that the maps $S_t$ are well-defined (at least under the assumption (\ref{eq:sec3-assumptions})). Moreover, it follows that $S_t$ are diffeomorphisms (e.g. \cite[Theorem 1.61]{GHLBookEdition2}), since the inverse maps $T_{t,t} = T_t := S_t^{-1}$ may be obtained by running the flow backwards:
\[
\frac{d}{d\tau} T_{t,\tau}(x) = -W_{t-\tau}(T_{t,\tau}(x)) ~,~ T_{t,0} = Id ~,~ \tau \in [0,t] ~.
\]
Clearly, the maps $S_t$ and $T_t$ inherit the symmetries of the vector field $W_t = - \nabla \log P_t^U(\exp(-V))$. As explained in the proof of Theorem \ref{thm:max}, $- \log P_t^U(\exp(-V))$ is invariant under the common symmetries of $U$ and $V$, i.e. our symmetry assumptions (\ref{eq:symmetry}), and so its gradient commutes with the group $O(E_1,\ldots,E_k)$ ; our maps therefore satisfy our symmetry assumptions as well.

Theorem \ref{thm:main2} guarantees that $D W_t \geq 0$ and hence $(D W_t)^* + D W_t\geq 0$ for every $t \geq 0$.
Consequently:
\[
 \frac{d}{dt} (DS_t)^*(x) DS_t(x) = (DS_t)^*(x) (D W_t)^*(S_t x) DS_t(x) + (DS_t)^*(x) D W_t(S_t x) DS_t(x) \geq 0 ~,
\]
and hence $(DS_t)^* DS_t \geq Id$ for every $t \geq 0$. In other words, $S_t$ is \emph{locally} an expansion. Since $S_t$ is also a diffeomorphism, it follows that it is in fact an expansion \emph{globally}. Indeed, $(DT_t)^* DT_t \leq Id$, which implies by integration and the triangle inequality that $|T_t(x) - T_t(y)| \leq |x-y|$.

\medskip
Next, we address the question of convergence of $\nu_t := P_t^U(\exp(-V)) \mu$ to $\mu$.
Although we will only require convergence in $L_1$ for the sequel, we state the following for completeness:
\begin{lem} \label{lem:convergence}
As $t \rightarrow \infty$, we have:
\begin{enumerate}
\item $P_t^U(\exp(-V)) \rightarrow 1$ in $L_p(\mu)$, for any $p \in [1,\infty)$.
\item $P_t^U(\exp(-V)) \rightarrow 1$ in $L_\infty(C)$, for any compact set $C \subset \Real^n$.
\item $\norm{\frac{d\nu_t}{dx} - \frac{d\mu}{dx}}_{L_p} \rightarrow 0$ for any $p \in [1,\infty]$.
\end{enumerate}
\end{lem}

For the proof, first recall that by (\ref{eq:L-prop}), $-L = -\Delta + \scalar{\nabla,\nabla U}$ is a symmetric positive semi-definite operator on the subspace $C^\infty(\Real^n) \cap L_2(\mu)$, and hence admits a Friedrichs extension to a self-adjoint positive semi-definite operator on a larger dense subspace $\D$ of $L_2(\mu)$, which we also denote by $-L$. Since $U$ is convex and $\mu = \exp(-U(x)) dx$ is a probability measure, it is known that $-L$ has a strictly positive spectral-gap $\lambda_1 > 0$ away from the trivial eigenvalue of $0$, corresponding to the constant functions: $\int - f L f d\mu \geq \lambda_1 \int f^2 d\mu$ for all $f \in \D_0 := \set{ f \in \D ; \int f d\mu = 0}$. For instance, by \cite{KLS} (see also \cite{EMilman-RoleOfConvexity}), one may estimate $\lambda_1 \geq c (\int |x| d\mu(x))^{-2} > 0$ for some universal numeric constant $c>0$.

\begin{proof}[Proof of Lemma~\ref{lem:convergence}]
Since $\lambda_1$ is strictly positive, the Spectral Theorem implies that $P_t^U(\exp(-V)) = \exp(-tL)(\exp(-V))$ tends in $L_2(\mu)$ to the projection of $\exp(-V)$ onto the constant functions, i.e. to the constant function $1 = \int \exp(-V) d\mu$.
Since $P_t^U$ is bounded in $L_\infty$ (as in Subsection \ref{subsec:smooth}), we deduce the first claim for $p \in [2,\infty)$ by interpolation (and by Jensen's inequality this extends to $p \in [1,\infty)$). Next, we follow an argument similar to that used by Ledoux \cite{LedouxLogSobLectureNotes}. Denoting $f = \exp(-V)$, write:
\[
 |P_t^U(f)(x) - 1| = |P_t^U(f)(x) - \int P_t^U(f)(y) d\mu(y)| \leq \int |P_t^U(f)(x) - P_t^U(f)(y)| d\mu(y) ~.
\]
Certainly $|P_t^U(f)(x) - P_t^U(f)(y)| \leq |\grad P_t^U(f)(z)| |x-y|$ for some intermediate point $z \in [x,y]$. But using that $U$ is convex, the following smoothing estimate is known (\cite{LedouxSpectralGapAndGeometry}):
\[
|\grad P_t^U(f)(z)| \leq \frac{1}{\sqrt{2t}} \norm{f}_{L_\infty} ~,
\]
and so:
\[
|P_t^U(f)(x) - 1| \leq \frac{1}{\sqrt{2t}} \norm{f}_{L_\infty} \brac{|x| + \int |y| d\mu(y)} ~.
\]
The uniform convergence (as $t \to \infty$) on compact subsets follows. Moreover, since $|x| \exp(-U(x))$ is necessarily bounded, we obtain the third claim for $p = \infty$.
The third claim for $p=1$ is equivalent to the first one with $p=1$, and so by interpolation, the third claim follows for all $p \in [1,\infty]$.
\end{proof}

Recall that a sequence of Borel measures $\set{\eta_k}$ is said to converge to a Borel measure $\eta$ \emph{weakly} (or in the weak$^*$-topology) if $\int \varphi d\eta_k \rightarrow \int \varphi d\eta$ for any bounded continuous test function $\varphi$; we will denote this by $\eta_k \rightharpoonup \eta$.
We define the $L_1$ distance between two absolutely continuous Borel measures $\eta_1,\eta_2$ on $\Real^n$ to be:
\[
d_{L_1}(\eta_1,\eta_2) := \int_{\Real^n} \abs{\frac{d\eta_1}{dx} - \frac{d\eta_2}{dx}} dx  ~;
\]
this coincides with the usual total-variation distance up to a factor of $2$. Clearly, convergence in $L_1$ implies weak convergence.

\begin{lem} \label{lem:approx}
Let $\set{\mu_k}$ and $\set{\nu_k}$ denote two sequences of absolutely continuous Borel measures on $\Real^n$, such that each $\nu_k$ is the push-forward of $\mu_k$ by a contracting map $T_k : \Real^n \rightarrow \Real^n$. Assume that $d_{L_1}(\mu_k,\mu) \rightarrow 0$ and $\nu_k \rightharpoonup \nu$.
Then there exists a contraction $T : \Real^n \rightarrow \Real^n$ pushing forward $\mu$ onto $\nu$. Moreover, any common symmetries possessed by $T_k$ are preserved by $T$.
\end{lem}

\begin{proof}
First, note that $T_k(0)$ must be uniformly bounded. Indeed, let $B(R_1)$ denote a ball around the origin so that $\mu(B(R_1)) \geq 3/4$.
The $L_1$ convergence immediately implies that $\mu_k(B(R_1)) \rightarrow \mu(B(R_1))$, and so $\mu_k(B(R_1)) \geq 2/3$ for large enough $k$. Similarly, if $B(R_2)$ denotes a ball so that $\nu(B(R_2)) \geq 3/4$, it follows easily from the weak convergence that $\nu_k(B(0,R_2)) \rightarrow \nu(B(R_2))$ (here we need to use the fact that the ball has finite perimeter and that our measures are absolutely continuous with respect to Lebesgue measure), and hence $\nu_k(B(R_2)) \geq 2/3$ for large enough $k$.
Consequently, for large enough $k$, $\mu_k(T_k^{-1}(B(R_2))) = \nu_k(B(R_2)) \geq 2/3$, and therefore $T_k^{-1}(B(R_2)) \cap B(R_1)$ is non-empty. Since $T_k$ is a contraction, it follows that $T_k(0) \in B(R_1+R_2)$.

Next, by passing to a subsequence if necessary, we may assume that $T_k(0)$ converges. Since $T_k$ are all contractions, and hence uniformly (Lipschitz) continuous, it follows by compactness and a standard diagonalization argument that, after passing to an appropriate subsequence, $T_k$ uniformly converges on compact subsets of $\Real^n$ to some map $T$, which is consequently a contraction, which preserves the common symmetries of $T_k$. It remains to show that $T$ pushes forward $\mu$ onto $\nu$.

This is equivalent to showing that $\int \varphi(Tx) d\mu(x) = \int \varphi(y) d\nu(y)$ for any bounded continuous test function $\varphi : \Real^n \rightarrow \Real$. Since by definition, for any $k$:
\[
 \int \varphi(T_k x) d\mu_k(x) = \int \varphi(y) d\nu_k(y) ~,
\]
and the right hand side converges to $\int \varphi(y) d\nu(y)$, it remains to show that the left hand side converges to $\int \varphi(Tx) d\mu(x)$. Indeed:
\begin{multline*}
\abs{ \int \varphi(T_k x) d\mu_k(x) - \int \varphi(Tx) d\mu(x) } \\
\leq
\abs{ \int \varphi(T_k x) d\mu_k(x) - \int \varphi(T_k x) d\mu(x) } +
\abs{ \int \varphi(T_k x) d\mu(x) - \int \varphi(T x) d\mu(x) } ~.
\end{multline*}
The first term on the right hand side converges to $0$ since $\varphi$ is bounded and $d_{L_1}(\mu_k,\mu) \rightarrow 0$. The second term converges to $0$ by Lebesgue's dominant convergence theorem, since (the bounded) $\varphi(T_k x)$ pointwise converges to $\varphi(Tx)$ (in fact uniformly on compact subsets). This concludes the proof.
\end{proof}

Lemma \ref{lem:convergence} (case (3) with $p=1$) ensures that $\nu_t$ converges to $\mu$ in $L_1$. Since $\nu$ is the push-forward of $\nu_t$ via $T_t$ which is a contraction, it follows by Lemma \ref{lem:approx} that there exists a contraction $T_\infty$ pushing forward $\mu$ onto $\nu$ and satisfying our symmetry assumptions. This concludes the proof of Theorem \ref{thm:main1} in the case that $U$ and $V$ are assumed smooth and under the additional assumptions of (\ref{eq:sec3-assumptions}). To conclude the theorem in the full generality, apply Lemma \ref{lem:approx} again to see that there exists a contraction pushing forward $\mu$ onto $\nu$, whenever these measures may be approximated by smooth measures satisfying the assumptions of the theorem and (\ref{eq:sec3-assumptions}). Such approximation is always possible by a standard argument: applying the Legendre transform to $V$, redefining the resulting function to be $+\infty$ beyond some large level, and applying the transform again, we obtain a convex Lipschitz function with the same symmetries, and it remains to convolve it with a smooth rotation-invariant mollifier, yielding the first part of (\ref{eq:sec3-assumptions}) ; a similar argument applies to the function $U$, whose special form \eqref{eq:U} reduces the approximation to an easy one-dimensional problem.
Lemma \ref{lem:approx} thus implies the general case of Theorem \ref{thm:main1}.

\section{Applications} \label{sec:apps}

The first application we would like to describe pertains to a generalization of the Gaussian Correlation Conjecture. This conjecture asks whether for any two convex subsets $A,B \subset \Real^n$, which are in addition centrally-symmetric ($C$ is called centrally-symmetric if $C = -C$), the following inequality is valid for the standard Gaussian measure $\gamma_n$ on $\Real^n$:
\begin{equation} \label{eq:GCC}
\gamma_n(A \cap B) \geq \gamma_n(A) \gamma_n(B) ~?
\end{equation}
We refer to \cite{SSZ-GaussianCorrelationConjecture,HargeGCCForEllipsoid,CorderoMassTransportAndGaussianInqs} and the references therein for the history of this conjecture, which remains open for $n \geq 3$. One of the most general results is due to Harg\'e \cite{HargeGCCForEllipsoid}, who confirmed the validity of (\ref{eq:GCC}) when one of the sets is a (centrally-symmetric) ellipsoid. This was subsequently given a different proof by Cordero-Erausquin \cite{CorderoMassTransportAndGaussianInqs}, as a direct corollary of Caffarelli's Contraction Theorem (in this context, it is worthwhile pointing out that
our construction of the expanding map $T^{-1}$ closely resembles Harg\'e's argument).
Replacing Caffarelli's theorem with Theorem \ref{thm:main1} in Cordero-Erausquin's argument, we obtain the following generalization:

\begin{cor} \label{cor:GCC}
Let $\mu = \exp(-U(x)) dx$ denote a probability measure on $\Real^n$ as in Theorem \ref{thm:main1}, which is in addition centrally symmetric (i.e. the quadratic part of $U$ on $E_0$ is assumed even). Let $B$ denote a centrally-symmetric \emph{convex} subset of $\Real^n$ satisfying the following symmetry assumptions:
\[
\exists C_B \subset \Real^{dim E_0 + k} \;\;\;\; \mathbf{1}_B(x) = \mathbf{1}_{C_B}(Proj_{E_0} x,|Proj_{E_1} x|,\ldots,|Proj_{E_k} x|) ~.
\]
Let $A$ denote a centrally-symmetric subset of $\Real^n$ so that, writing for $x \in \Real^n$, $x = (x_0,x_1,\ldots,x_k)$ with $x_i \in E_i$, we have:
\begin{multline} \label{eq:A-prop}
\text{ if } (x_0,x_1,\ldots,x_k) \in A \text{ then } \\
 \forall y_0 \in E_0, \;\; \norm{y_0}_{\E} \leq \norm{x_0}_{\E}, \; \forall t_i \in [-1,1], \text{ we have } (y_0, t_1 x_1, \ldots, t_k x_k) \in A  ~,
\end{multline}
where $\norm{\cdot}_{\E}$ is the norm associated with some centrally-symmetric ellipsoid $\E \subset E_0$.
Then:
\[
\mu(A \cap B) \geq \mu(A) \mu(B) ~.
\]
\end{cor}
Clearly, this generalizes the result of Harg\'e and Cordero--Erausquin, by choosing $\mu = \gamma_n$ and $E_0 = \Real^n$.
\begin{proof}
First, by applying an appropriate linear transformation $P$ in $E_0$ which leaves the orthogonal complement invariant, we may assume that $\E$ is a Euclidean ball in $E_0$, since $P(B)$ and $P_*(\mu)$ continue to satisfy the assumptions of the theorem (indeed, $P$ only affects the quadratic part of $U$, which remains quadratic and even). Defining the probability measure $\mu_B$ as the restriction of $\mu$ onto $B$, i.e. $\mu_B(C) = \mu(C \cap B) / \mu(B)$, our task is to show that $\mu_B(A) \geq \mu(A)$. It is standard to approximate $\mathbf{1}_B / \mu(B)$ in $L_1(\Real^n)$ by functions of the form $\exp(-V_k)$, where $V_k$ is convex and satisfies the same symmetries as $B$, implying that $\exp(-V_k) \mu$ tends to $\mu_B$ in total-variation. Applying Theorem \ref{thm:main1} and Lemma \ref{lem:approx}, we deduce that there exists a contraction $T$ pushing forward $\mu$ onto $\mu_B$ and satisfying our symmetry assumptions. Since $T$ commutes with $O(E_1,\ldots,E_k)$, it follows easily that $Proj_{E_i} T(x)$ is radial for $i=1,\ldots,k$:
\begin{equation} \label{eq:radial-sym}
 Proj_{E_i} T(x) = T_i(x_0,|x_1|,\ldots,|x_k|) \frac{x_i}{|x_i|} \text{ if $x_i \neq 0$ and 0 otherwise} ~.
\end{equation}
Moreover, since $B$ and $\mu$ were assumed centrally-symmetric, it is easy to check that $T$ will also preserve this additional symmetry. Denoting by $R_i$ the reflection in the subspace $E_i$, i.e. $R_i(x) = x - 2Proj_{E_i} x$ for $i=0,1,\ldots,k$, we conclude that $T$ commutes with all the $R_i$'s.

It remains to note that $T(A) \subset A$. Indeed, using the commutation with $R_i$ and the contraction property of $T$, we have:
\[
 2 |Proj_{E_i} T(x)| = |R_i (T(x)) - T(x)| = | T(R_i(x)) - T(x) | \leq |R_i(x) - x| = 2 |Proj_{E_i} x| ~,
\]
and so $|Proj_{E_i} T(x)| \leq |Proj_{E_i} x|$ for $i=0,1,\ldots,k$. Together with (\ref{eq:radial-sym}) and the symmetries (\ref{eq:A-prop}) of $A$, it follows that $T(A) \subset A$. Consequently $A \subset T^{-1}(A)$, and therefore:
\[
\frac{\mu(A \cap B)}{\mu(B)} = \mu_B(A) = \mu(T^{-1}(A)) \geq \mu(A) ~.
\]
The proof is complete.
\end{proof}
\begin{rem}
It is possible to replace the requirement $t_i \in [-1,1]$ in (\ref{eq:A-prop}) by $t_i \in [0,1]$. This is achieved by using the Brenier map $T_{opt}$ of Theorem \ref{thm:main3} instead of $T$ in the proof above, thereby ensuring that the $\set{T_i}_{i=1}^k$ in (\ref{eq:radial-sym}) are always non-negative, as explained in Section \ref{sec:revisit}.
\end{rem}

The following two additional corollaries may be easily obtained from the previous one by integration by parts:

\begin{cor}
Let $\mu$ denote a probability measure on $\Real^n$ as in Theorem \ref{thm:main1}, which is in addition centrally symmetric. Let $f,g : \Real^n \rightarrow \Real_+$ denote two measurable bounded functions, so that for each $a,b > 0$, the level sets $f^{-1}([a,\infty))$ and $g^{-1}([b,\infty))$ satisfy the assumptions on the sets $A$ and $B$ in Corollary \ref{cor:GCC}, respectively. Then:
\[
 \int f g d\mu \geq \int f d\mu \int g d\mu ~.
\]
\end{cor}

\begin{cor}
Let $\mu,\nu$ denote two probability measures as in Theorem \ref{thm:main1}, and assume in addition that both are centrally symmetric. Let
$\Gamma : \Real^n \rightarrow \Real_+$ denote a measurable function such that all of its level sets $\Gamma^{-1}([0,a])$ (individually) satisfy the assumption on the set $A$ in Corollary \ref{cor:GCC}. Then:
\[
\int \Gamma(x) d\nu(x) \leq \int \Gamma(x) d\mu(x) ~.
\]
\end{cor}

These corollaries generalize the correlation inequalities obtained in \cite{BrascampLiebPLandLambda1,HargeGCCForEllipsoid,CaffarelliContraction} for the case $dim E_0 = n$.
We remark that when $dim E_0 = 0$, the corollaries may be obtained directly without appealing to Theorem \ref{thm:main1}, so the more interesting case is when $0 < dim E_0 < n$.

\medskip

Finally, we also mention that contracting maps constitute a very useful tool to transfer isoperimetric inequalities from one measure-metric space to another.
Note that the measure $\mu$ of Theorem \ref{thm:main1} is a product measure, with each factor being either a Gaussian or a log-concave radially symmetric measure. The isoperimetric inequality satisfied by the former factor is well known \cite{SudakovTsirelson,Borell-GaussianIsoperimetry}, and has recently been identified (up to numeric constants) for the latter factor \cite{HuetSphericallySymmetric}. The tools to transfer these inequalities to the product measure have also recently been obtained \cite{BartheTensorizationGAFA,BCRHard,BCRSoft,EMilmanRoleOfConvexityInFunctionalInqs}, and so consequently, the isoperimetric inequality satisfied by $\mu$ is well understood. Using the contracting map $T$ of Theorem \ref{thm:main1}, it follows that the same isoperimetric inequality is satisfied by the measure $\nu$. We refer to \cite{LatalaJacobInfConvolution} for further examples of using contracting maps to transfer isoperimetric inequalities, and for further information.

\section{Caffarelli's proof revisited} \label{sec:revisit}

Let us now sketch the proof of Theorem \ref{thm:main3}, which is based on the proof of \cite[Theorem 11]{CaffarelliContraction}, but requires
an additional ingredient from \cite{CaffarelliContraction} in the form of Theorem \ref{thm:ingred} below. Throughout this section we use $T$ to denote the Brenier optimal-map.

\subsection{The Radial Case}

We begin with the elementary case when $\mu = \exp(-\rho(|x|)) dx$ and $\nu = \exp(-(\rho+v)(|x|)) dx$ are radial. This case does
not require the use of Theorem \ref{thm:ingred}, and as we will see, clearly motivates the condition $\rho''' \leq 0$ in Theorem \ref{thm:main1}.

First, it is immediate to reduce to the one dimensional case, when $\mu$ and $\nu$ are supported on $\Real_+$. Indeed, by the radial symmetry and the uniqueness of the Brenier map $T = \nabla \varphi$ with $\varphi: \Real^n \rightarrow \Real$ a convex function, it follows that $T$ must also be radially symmetric, i.e. commute with the orthogonal group. Consequently, we may write $\varphi(x) = \phi(|x|)$ with $\phi: \Real_+ \rightarrow \Real$ convex, and $T(r \theta) = T_1(r) \theta$ for $\theta \in S^{n-1}$ and $r \in \Real_+$. $T_1 = \phi' : \Real_+ \rightarrow \Real_+$ is precisely the Brenier map pushing forward $\exp(-\rho(r)) r^{n-1} dr$ onto $\exp(-(\rho(r)+v(r))) r^{n-1} dr$. Denoting $\rho_1(r) = \rho(r) - (n-1) \log r$, we see that $\rho_1$ remains convex and $\rho_1''' \leq 0$, and so it is enough to show that when in addition $v : \Real_+ \rightarrow \Real$ is convex and non-decreasing, the Brenier map $T_1$ pushing forward $\mu_1 = \exp(-\rho_1(r)) dr$ onto $\nu_1 =\exp(-(\rho_1(r)+v(r))) dr$ is a contraction.

Indeed, in the one dimensional case, the derivative of a convex function is simply a monotone non-decreasing one, and so the Brenier map is the unique non-decreasing map pushing forward $\mu_1$ onto $\nu_1$, given by:
\begin{equation} \label{eq:1D-pf}
\int_0^{T_1(x)} \exp(-(\rho(r)+v(r))) dr = \int_0^x \exp(-\rho(r)) dr ~.
\end{equation}
Since $\rho,v$ are assumed smooth enough, so is $T_1$. Taking derivatives, we obtain:
\begin{equation} \label{eq:1D-MA}
\log T_1'(x) = -\rho(x) + \rho(T_1(x)) + v(T_1(x)) ~.
\end{equation}
Assume that the maximum of $T_1'$ is attained at $x_0 \in \Real_+$. To ensure this, one would actually need to restrict $\nu_1$ onto a compact subset, in which case $\lim_{x \rightarrow \infty} T_1'(x) = 0$ and so the (positive) maximum is attained, and conclude with an approximation argument (as in \cite{CaffarelliContraction}) ; we omit the details here. Our task is to show that $T_1'(x_0) \leq 1$. If $x_0 = 0$, since $T_1(0) = 0$ and $\exp(-v(0)) \geq 1$ (otherwise $\mu$ and $\nu$ could not both have total mass 1), it follows that $T_1'(0) \leq 1$, as required. Otherwise, denoting $F = \log T_1'$, since $F$ and $T_1'$ have a local maximum at $x_0$, it follows that $T''_1(x_0) = 0$ and that:
\begin{eqnarray*}
\!\!\!\!\!\!\!\!
0 & \geq & F''(x_0) \\
& = & -\rho_1''(x_0) + (T_1'(x_0))^2 (\rho_1''(T_1(x_0)) + v''(T_1(x_0)))  + T_1''(x_0)(\rho_1'(T_1(x_0)) + v'(T_1(x_0))) \\
& = & -\rho_1''(x_0) + (T_1'(x_0))^2 (\rho_1''(T_1(x_0)) + v''(T_1(x_0))) ~.
\end{eqnarray*}
Since $v'' \geq 0$ and $\rho_1'' \geq 0$, we obtain that:
\begin{equation} \label{eq:compare}
(T_1'(x_0))^2 \leq \frac{\rho_1''(x_0)}{\rho_1''(T_1(x_0))} ~.
\end{equation}
In Caffarelli's argument, $\rho_1$ is a quadratic polynomial, and therefore the right-hand side above is identically $1$. However, since $T_1(x) \leq x$ for all $x \in \Real_+$, as easily verified from (\ref{eq:1D-pf}) and the fact that $v$ is non-decreasing, we obtain by the mean-value theorem that the right-hand side is not greater than 1 as soon as $\rho_1''' \leq 0$. This concludes the proof and explains the latter condition.

\medskip

We remark that in this simple case, the Brenier map and the map we construct in our proof of Theorem \ref{thm:main1} do in fact coincide, since the latter one is
also radially symmetric, and is constructed as a limit of diffeomorphisms, and hence must be monotone on each ray from the origin.

\subsection{The General Case}

Let $\mu = \exp(-U(x)) dx$ and $\nu = \exp(-(U(x) + V(x))) dx$ be two probability measures in $\Real^n$, satisfying the assumptions in Theorem \ref{thm:main1}.
We will actually assume that $\nu$ is supported on a compact convex set $C$, to be specified later on, and that $U \in C^{3,\alpha}(\Real^n)$, $U$ is strictly convex, and $V \in C^{3,\alpha}(C)$ ; the general case follows by a standard approximation argument, under which one may show that the corresponding Brenier maps converge to the gradient of a convex function, i.e. the Brenier map for the limiting measures, and the contraction property is trivially preserved in the limit.

Let $T = \nabla \varphi$ denote the Brenier map pushing forward $\mu$ onto $\nu$, where $\varphi : \Real^n \rightarrow \Real$ is a convex potential.
It follows from our assumptions and Caffarelli's regularity theory \cite{CaffarelliStrictlyConvexIsHolder,CaffarelliHigherHolderRegularity,CaffarelliRegularity} that $\varphi \in C^{5,\alpha}_{loc}(\Real^n)$.
It also follows from the proof of \cite[Lemma 4]{CaffarelliContraction} and the subsequent remark that $\norm{D T}(x) = \max_{\xi \in S^{n-1}} D^2_{\xi,\xi} \varphi(x)$ attains a maximum in $\Real^n$, since $D^2_{\xi,\xi} \varphi(x)$ tends to $0$ as $|x| \rightarrow \infty$ uniformly in $\xi \in S^{n-1}$, when $C$ is convex. We will denote by $x_0$ a point where this maximum is attained. Our task is to show that $\norm{T}_{Lip} := D^2_{e,e} \varphi(x_0) \leq 1$, where $e \in S^{n-1}$ is the eigenvector of $D^2 \varphi(x_0)$ corresponding to its maximal eigenvalue, and hence:
\begin{equation} \label{eq:e-max}
D_e D \varphi(x_0) = D^2_{e,e} \varphi(x_0) e ~.
\end{equation}
As usual, attaining the maximum at $x_0$ implies that:
\begin{equation} \label{eq:MA-max}
\nabla D^2_{e,e} \varphi(x_0) = D^2_{e,e} T(x_0) = 0 ~ , ~  D^2 D^2_{e,e} \varphi(x_0) = D^2_{e,e} DT(x_0) \leq 0 ~.
\end{equation}

As in (\ref{eq:1D-MA}), the change-of-variables formula resulting from the definition of push-forward is:
\begin{equation} \label{eq:MA}
\log \det DT(x) = -U(x) + U(T(x)) + V(T(x)) ~.
\end{equation}
Differentiating (\ref{eq:MA}) twice in the direction of $e$, we obtain:
\begin{eqnarray} \label{eq:bigf}
\!\!\!\!\!\! & & - tr( (DT)^{-1}(x) D_e DT(x) (DT)^{-1}(x) D_e DT(x) ) + tr( (DT)^{-1}(x) D^2_{e,e} DT(x) )\\
\nonumber \!\!\!\!\!\! & = & -D^2_{e,e} U(x) + \scalar{ D^2 (U+V) (T(x)) D_e T(x), D_e T(x)} + \scalar{D (U+V)(T(x)), D^2_{e,e} T(x) } ~.
\end{eqnarray}
Using that $DT = D^2 \varphi > 0$, observe that $D_e DT (DT)^{-1} D_e DT \geq 0$. Recalling by (\ref{eq:MA-max}) that $D^2_{e,e} DT(x_0) \leq 0$, and using the fact that $tr(AB) \geq 0$ if $A,B \geq 0$, it follows that the left-hand side of (\ref{eq:bigf}) is non-positive when evaluated at $x_0$. Noting by (\ref{eq:MA-max}) that the last summand on the right-hand side of (\ref{eq:bigf}) vanishes at this point, and using $D^2 V \geq 0$ and (\ref{eq:e-max}), we conclude that:
\[
D^2_{e,e} U(x_0) \geq \scalar{ D^2 U(T(x_0)) D_e D \varphi (x_0), D_e D \varphi (x_0) } = D^2_{e,e} U(T(x_0)) |D^2_{e,e} \varphi(x_0)|^2 ~.
\]
Since $D^2 U > 0$, we obtain the analogue of (\ref{eq:compare}):
\[
\norm{T}_{Lip}^2 = |D^2_{e,e} \varphi(x_0)|^2 \leq \frac{D^2_{e,e} U(x_0)}{D^2_{e,e} U(T(x_0))} ~.
\]
When $U$ is quadratic, this is already enough to guarantee that $T$ is contracting. To make sure that the right-hand side is not greater than 1 under more general circumstances, we would need by the mean-value theorem to ensure that:
\begin{equation} \label{eq:ensure}
 \left . (D^3 U) \right |_{y}(e,e,x_0 - T(x_0)) \leq 0 \;\;\; \forall y \in [x_0,T(x_0)] ~.
\end{equation}

By the uniqueness of the Brenier map and the symmetries of $\mu$ and $\nu$, we know that $T$ must satisfy our symmetry assumptions. Consequently, as in the proof of Corollary \ref{cor:GCC}, $T$ must act radially on each $E_i$, $i=1,\ldots,k$:
\[
Proj_{E_i} T(x) =\begin{cases}
  T_i(Proj_{E_0} x,|Proj_{E_1} x|,\ldots,|Proj_{E_k} x|) \frac{Proj_{E_i} x}{|Proj_{E_i} x|} & \text{ if $Proj_{E_i} x \neq 0$}, \\
    0   & \text{otherwise}.
\end{cases}
\]
As the gradient of a convex function, we must have $\scalar{T(x) - T(y) , x - y} \geq 0$ for all $x,y \in \Real^n$, and using $y = x - 2 Proj_{E_i} x$ (reflecting $x$ in $E_i$ about the origin) implies that necessarily $T_i \geq 0$. Consequently:
\[
\forall i=1,\ldots,k \;\;\; \exists a_i(x_0) \geq 0 \;\;\; Proj_{E_i} T(x_0) = a_i(x_0) Proj_{E_i} x_0 ~.
\]
We conclude from Lemma \ref{lem:geom1} that (\ref{eq:ensure}) would follow if we could show that:
\begin{equation} \label{eq:a}
\forall x \in \Real^n \;\;\; \forall i = 1,\ldots,k \;\;\; a_i(x) \leq 1 ~.
\end{equation}
Geometrically, this means we that we have reduced the task of showing that $T$ is a contraction, to showing that $T$ is a contraction \emph{with respect to the origin} on each $E_i$. Note that in the radial case, this followed trivially from the monotonicity of $v$.

\medskip

To show (\ref{eq:a}), we require
the following additional ingredient \cite[Theorem 6]{CaffarelliContraction}.
\begin{thm}[Caffarelli] \label{thm:ingred}
Let $U_1 \in C^{1,\alpha}(\Omega_1)$ and $U_2 \in C^{1,\alpha}(\Omega_2)$, where $\Omega_2 = \times_{i=1}^n [a_i,b_i] \subset \Real^n$ and $\Omega_1 \supset \Omega_2$, so that $\int_{\Omega_i} \exp(-U_i(x)) dx = 1$.
Let $\tilde{T}$ denote the Brenier optimal-transport map pushing forward $\exp(-U_1(x))dx$ onto $\exp(-U_2(x))dx$, and let $S$ denote
a fixed subset of the coordinates $\set{1,\ldots,n}$. Assume that for any $x \in \Omega_1$, $y \in \Omega_2$ and $j \in S$:
\begin{equation} \label{eq:ingred-cond}
\text{$\forall i \in S \;\;\; y_i \leq x_i$ and $x_j = y_j$} \; \Rightarrow \; \frac{d}{dx_j} U_1 (x) \leq \frac{d}{dy_j} U_2(y) ~.
\end{equation}
Then $\tilde{T}(x)_i \leq x_i$ for all $i\in S$, for any $x \in \Omega_1$.
\end{thm}

In our formulation, we have exchanged between source and target measures (using that the Brenier map in this case is precisely the inverse of the original one), removed the assumption that $\Omega_1 = \Omega_2$, and consider only a subset of the coordinates for which the assumption and conclusion hold (as can be easily verified by inspecting the proof).

\medskip

Fix a coordinate structure determined by our decomposition of $\Real^n$ into $E_i$, let $Q$ denote the set of coordinates corresponding to $E_0$, and let $S$ denote the set of all other coordinates, corresponding to the subspaces $E_1,\ldots,E_k$. Set $C = [-R,R]^n$, $\Omega_1 = \Real^Q \times \Real_+^S$, $\Omega_2 = [-R,R]^Q \times [0,R]^S$, $U_1 = U + c_1$ and $U_2 = U + V + c_2$, where $c_i$ are constants designed to make $\exp(-U_i(x))dx$ probability measures on $\Omega_i$. The symmetries of
$T$ described above imply that it is enough to verify (\ref{eq:a}) for $x \in \Omega_1$ and that $T|_{\Omega_1} = \tilde{T}$, where $\tilde{T}$ is given by Theorem \ref{thm:ingred}. Consequently, the desired (\ref{eq:a}) will follow from the conclusion of Theorem \ref{thm:ingred} if we verify (\ref{eq:ingred-cond}).

Fix $j \in S$, corresponding to a subspace $E_l$. Lemma \ref{lem:geom2} implies that $\frac{d}{dy_j} V(y) \geq 0$ for any $y \in \Omega_2$, and so it is enough to verify that for $x \in \Omega_1$ and $y \in \Omega_2$:
\begin{equation} \label{eq:final-red}
\text{$\forall i \in S \;\;\; y_i \leq x_i$ and $x_j = y_j$} \; \Rightarrow \; \frac{d}{dx_j} U(x) \leq \frac{d}{dy_j} U(y) ~.
\end{equation}
But $\frac{d}{dx_j} U(x) = \frac{\rho_l'(|Proj_{E_l} x|)}{|Proj_{E_l} x|} x_j$, and when $x_j$ is fixed, the coefficient in front of it is non-increasing in $|Proj_{E_l} x|$ since $\rho_l'$ was assumed concave and $\rho_l'(0) = 0$. Since $x \in \Omega_1$ and $y \in \Omega_2$, the assumption $y_i \leq x_i$ for all $i \in S$ implies that $|Proj_{E_l} y| \leq |Proj_{E_l} x|$, confirming the desired (\ref{eq:final-red}). This finally concludes the proof.

\section{Comparing the two maps} \label{sec:comparing}
In this section, we compare the map $T$ (as constructed in Subsection~\ref{SS:construction of T}) with the Brenier map $T_{opt}$.

First, it is natural to ask whether the two maps $T$ and $T_{opt}$ coincide, at least under the assumptions of Theorem~\ref{thm:main1}. To analyze this question, recall that $S_t$ was constructed as follows:
\begin{equation} \label{eq:W-eq-again}
\frac{d}{dt} S_t(x) = W_t(S_t(x)) ~,~ S_0 = Id ~,~ \text{with} ~~ W_t := \nabla Z_t ~,~ Z_t := -\log P_t^U(\exp(-V)) ~.
\end{equation}
Denoting $B_t(x) := D^2 Z_t(S_t(x))$ and taking spatial derivatives, we obtain:
\begin{equation} \label{eq:DW-eq-again}
\frac{d}{dt} DS_t(x) = B_t(x) DS_t(x) ~,~ DS_0(x) \equiv Id ~.
\end{equation}
As is well known, a necessary and sufficient condition for being the gradient of a function on a simply connected domain, is having a symmetric derivative tensor. It follows that:
\begin{equation} \label{eq:commutation}
\text{if all of $\set{B_t}_{t \geq 0}$ commute with each other} ~,
\end{equation}
ensuring that $DS_t$ remains symmetric along the flow, then we can conclude that $S_t$ is the gradient of some function (for each $t$). Moreover, we could then write:
\[
DS_t(x) = \exp\brac{\int_0^t B_s(x) ds } ~,
\]
from which it would follow that $DS_t$ is pointwise positive semi-definite, and hence $S_t$ must be the gradient of a \emph{convex} function. The inverse map $T_t = S_t^{-1}$ would then be the gradient of a convex function as well, and this property may be shown to be preserved in the limit as $t \rightarrow \infty$, obtaining the Brenier map transporting $\mu = (S_\infty)_*(\nu)$ onto $\nu$.

Condition (\ref{eq:commutation}) implies that in all one-dimensional situations ($n=1$ or radially symmetric data), both maps $T$ and $T_{opt}$ \emph{do} coincide. However, it is easy to check that generically, the sufficient condition (\ref{eq:commutation}) will be severely \emph{violated}, for instance by constructing examples (see below) so that for some $x$:
\begin{equation} \label{eq:no-commutation2}
0 \neq [\frac{d}{dt} B_t(x), B_t(x) ] = [ D^2 \frac{d}{dt} Z_t + D^3 Z_t D Z_t, D^2 Z_t ](S_t(x)) \;\;\; \forall t \geq 0 ~,
\end{equation}
where $[A,B] = AB - BA$ denotes the Lie bracket. Moreover, it is not hard to show that (\ref{eq:no-commutation2}) implies that $DS_t(x)$ is non-symmetric on some non-empty interval $t \in (0,t_0)$, and that for any non-empty interval $(t_1,t_2) \subset (0,\infty)$, $\set{DS_t(x)}_{t \in (t_1,t_2)}$ cannot all commute. In other words, the path of diffeomorphisms $[0,\infty) \ni t \mapsto S_t$ will generically not coincide with the path of optimal interpolating maps $[0,1) \ni s \mapsto (1-s) Id + s S_{opt}$, where $S_{opt} = T_{opt}^{-1}$ is the Brenier map pushing forward $\nu$ onto $\mu$, and in fact the set of times $t$ where these two paths intersect will be discrete.

All of this suggests that generically, the lack of symmetry (or path separation) should persist in the limit as $t \rightarrow \infty$, and hence that the limiting map $T$ should be different from $T_{opt}$. However, although we believe that (\ref{eq:commutation}) is actually also a necessary condition (at least generically) for obtaining the Brenier map, we are unable to rule out the possibility that the symmetry may be recovered in the limit. In particular, we are unable to show that the two maps are different even for the following simple example, where (almost) everything may be explicitly computed:

\begin{example}
Let $U,V$ be given by:
\begin{align*}
& U(x) = \frac{1}{2}\scalar{A x,x} ~,~ V(x) = \frac{1}{2}\scalar{B x,x} ~, \\
& \text{$A,B$ are positive-definite \emph{non-commuting} matrices} ~,
\end{align*}
and set:
\[
\mu = c_1 \exp(-U(x))dx ~,~ \nu = c_2 \exp(-(U(x) + V(x))) dx ~,
\]
with $\set{c_i}$ chosen so that the resulting measures have total mass $1$.

It is easy to see (e.g. \cite[Example 1.7]{McCannConvexityPrincipleForGases}) that $T_{opt}$ is a linear map given by the positive-definite matrix $C_{opt} = A^{1/2} (A^{1/2} (A+B) A^{1/2})^{-1/2} A^{1/2}$. The Mehler formula \cite{HargeGCCForEllipsoid} for an affine Ornstein-Uhlenbeck diffusion implies that the tensor $D^2 Z_t = -D^2 \log P_t^U(\exp(-V))$ is an explicitly computable fixed matrix $M_t$ for every time $t \geq 0$, and so by (\ref{eq:DW-eq-again}), the flow maps $\set{S_t}$ are also linear, given by a family of matrices $\set{L_t}$. Moreover, $L_t$ satisfy an explicit matrix-valued ODE, and one may also show that $L_t^* (A+M_t) L_t = A+B$.
The resulting map $T$ is then the linear map given by the matrix $L_\infty^{-1}$, where $L_\infty = \lim_{t \rightarrow \infty} L_t$.

Showing that $T \neq T_{opt}$ when $A,B$ do not commute then amounts to proving that $L_\infty$ is not symmetric in this case; we were unable to verify this. When $A,B$ do commute, then so do all the matrices $\set{M_t}$, so (\ref{eq:commutation}) is satisfied and $T = T_{opt}$.
\end{example}

An additional aspect of comparing between $T$ and $T_{opt}$ pertains to the condition that $U$ be convex in Theorems \ref{thm:main1} and \ref{thm:main3}. This condition was absolutely crucial in Caffarelli's argument and the proof of Theorem \ref{thm:main3}. However, an inspection of the proof of Theorem \ref{thm:main1} reveals that this condition was only used in the proof of (\ref{eq:Z-PDE}), (\ref{eq:DV-bound}) and Lemma \ref{lem:convergence}, and it is actually possible to relax our condition to $D^2 U \geq -c Id$ by a careful adaptation of the arguments (and in particular, avoid using the Spectral Theorem, since $-L$ will no longer have a spectral gap).
Unfortunately, the convexity of $U$ actually follows from the other assumptions of Theorem \ref{thm:main1}, namely that $\mu = \exp(-U(x)) dx$ has finite total mass and that $\rho'''_i \leq 0$ on $\Real_+$, so ultimately there is no real gain here in using $T$ over $T_{opt}$. But this difference in the significance of the convexity of $U$ to the proof, perhaps reinforces the intuition that these two maps should be (generically) different.

\medskip

Before concluding, we mention a couple of advantages of working with the map $T$ over the Brenier map $T_{opt}$.
In the proof of the contraction property of $T_{opt}$ (Theorem \ref{thm:main3}), Caffarelli's regularity theory for the fully-nonlinear Monge--Amp\`ere equation was an essential ingredient. In contrast, in our study of the map $T$ (Theorem \ref{thm:main1}), we only employed the classical regularity results for linear parabolic PDEs. This lends our heat-diffusion construction to further generalizations, in situations where the regularity for the Monge--Amp\`ere equation and the Brenier--McCann optimal-transport map has yet to be established, or alternatively is known to be false, for instance in the Riemannian-manifold setting (see \cite{VillaniOldAndNew}).
In addition, other choices for the driving potential $Z_t$ in our flow scheme (\ref{eq:W-eq-again}) are also possible, in accordance to the property one wished to establish.

\setlinespacing{0.93}
\setlength{\bibspacing}{2pt}

\bibliographystyle{plain}

\def\cprime{$'$}

\end{document}